\documentclass[a4paper]{amsart}

\usepackage{fullpage}
\usepackage{amssymb,amsmath,array,amsthm,amsfonts,amscd}
\usepackage{version}
\usepackage{algpseudocode}
\usepackage{tikz}
\usepackage{mathrsfs}
\usepackage{url}
\usepackage{enumitem}

\usepackage{caption}
\usepackage{color}
\usepackage[all]{xypic}

\usepackage[colorlinks,
             linkcolor=blue,
             citecolor=green,
             pdfpagemode=None,
             bookmarksopen=true
             bookmarksnumbered=true]{hyperref}

\allowdisplaybreaks[1]

\numberwithin{equation}{section}

\theoremstyle{plain}
\newtheorem{theorem}[equation]{Theorem}
\newtheorem{proposition}[equation]{Proposition}
\newtheorem{corollary}[equation]{Corollary}
\newtheorem{lemma}[equation]{Lemma}

\theoremstyle{definition}

\newtheorem{defn}[equation]{Definition}

\newtheorem{example}[equation]{Example}

\theoremstyle{definition}

\newtheorem{ex}[equation]{Example}

\theoremstyle{remark}
\newtheorem{remark}[equation]{Remark}
\newtheorem{rem}[equation]{Remark}

\newcommand{\PP}{\mathbb{P}}
\newcommand{\N}{\mathbb{N}}
\newcommand{\Z}{\mathbb{Z}}
\newcommand{\Q}{\mathbb{Q}}
\newcommand{\C}{\kk}

\newcommand{\al}{\alpha}

\newcommand{\de}{\delta}
\newcommand{\la}{\lambda}

\newcommand{\F}{\mbox{\rm Frac\,}}
\newcommand{\Aut}{\mbox{\rm Aut\,}}
\newcommand{\PAut}{\mbox{\rm PAut\,}}

\newcommand{\GL}{\mbox{\rm GL\,}}

\newcommand{\Lie}{\mbox{\rm Lie\,}}

\newcommand{\td}{\mbox{\rm trdeg\,}}

\newcommand{\ff}{\footnote}
\newcommand{\DMO}{\DeclareMathOperator}

\newcommand{\beq}{\begin{equation}}
\newcommand{\eeq}{\end{equation}}
\newcommand{\mb}{\mathbb}
\newcommand{\mf}{\mathfrak}
\newcommand{\germ}{\mathfrak}
\newcommand{\kk}{{\Bbbk}}
\renewcommand{\AA}{\mb A}
\newcommand{\CC}{\kk}
\newcommand{\NN}{\mb N}
\newcommand{\QQ}{\mb Q}
\newcommand{\ZZ}{\mb Z}
\newcommand{\mc}{\mathcal}
\newcommand{\sC}{\mc C}

\newcommand{\del}{\partial}
\newcommand{\ssm}{\smallsetminus}
\newcommand{\ang}[1]{\langle #1 \rangle}

\DMO{\Ann}{Ann}
\DMO{\Pspec}{PSpec}
\DMO{\Spec}{Spec}
\DMO{\Kdim}{Kdim}
\DMO{\gr}{gr}
\DMO{\rank}{rank}
\DMO{\hilb}{hilb}
\DMO{\gldim}{gldim}
\DMO{\Span}{span}
\DMO{\Ext}{Ext}
\DMO{\im}{Im}
\DMO{\Der}{Der}
\DMO{\Proj}{Proj}
\DMO{\Frac}{Frac}
\DMO{\Sym}{Sym}

\newcommand{\st}{\left\vert\right.}

\renewcommand{\implies}{\Rightarrow}
\newcommand{\sO}{\mc O}

\newcommand{\PGL}{\mb{P}GL}

\title{A new family of Poisson algebras and their deformations}
\author{C. Lecoutre and S. J. Sierra}
\address{(Lecoutre) School of Mathematics, Statistics and Actuarial Science,
University of Kent, Canterbury CT2 7NF, UK.}
\email{C.Lecoutre@kent.ac.uk}
\address{(Sierra) School of Mathematics, Peter Guthrie Tait Road, King's Buildings,
University of Edinburgh, Edinburgh EH9 3FD, UK.}
\email{s.sierra@ed.ac.uk}
\thanks{The first author was partially supported by an LMS mobility grant.  The second author was  partially   supported by  EPSRC grant EP/M008460/1.}
\date{\today}

\keywords{Artin-Schelter regular, Calabi-Yau algebra, Poisson algebra, semiclassical limit}
\subjclass[2010]{Primary:  16S38; Secondary 14A22, 17B63, 16W70, 16S80}
\begin{document}

\begin{abstract}

Let $\kk$ be a field of characteristic zero. 
For any positive integer $n$ and any scalar $a\in\kk$, we construct a family of Artin-Schelter regular algebras $R(n,a)$, which are quantisations of Poisson structures on $\kk[x_0,\dots,x_n]$.  This generalises an example given by Pym when $n=3$.  
For a particular choice of the parameter $a$ we obtain new examples of Calabi-Yau algebras when $n\geq 4$. 
We also study the ring theoretic properties of the algebras $R(n,a)$.
We show that  the point modules of $R(n,a)$ are parameterised by a bouquet of rational normal curves in $\PP^{n}$, and that the prime spectrum of $R(n,a)$ is homeomorphic to the Poisson spectrum of its semiclassical limit. 
Moreover, we explicitly describe $\Spec R(n,a)$ as a union of commutative strata. 
\end{abstract}

\maketitle



\section{Introduction} \label{INTRO}

One of the fundamental notions of noncommutative algebraic geometry is that the noncommutative analogues of polynomial rings (or, alternatively, the coordinate rings of ``noncommutative projective spaces'') are algebras which are {\em Artin-Schelter regular} \cite{AS}.  
Constructing and, if possible, classifying, Artin-Schelter regular algebras of various global dimensions is thus one of the core problems in the subject.  
Within the class of quadratic Artin-Schelter regular algebras one also seeks to construct those which are {\em Calabi-Yau} in the sense of Ginzburg \cite{Ginzburg}.  
In general, the problem of classifying Artin-Schelter regular algebras (or Calabi-Yau algebras) is unsolved, even for global dimension 4.

In 2014, Brent Pym classified  4-dimensional Calabi-Yau algebras arising as deformation-quantisations of torus-invariant Poisson structures on $\AA^4_{\mb C}$.  He showed that there are six families of Calabi-Yau algebras which arise this way; the most interesting example has generators $x_0, x_1, x_2, x_3$ and relations
\beq\label{pym}
\begin{split}
[x_0, x_1] &  = 5 x_0^2 \\
[x_0,x_2] &= -\frac{45}{2} x_0^2 + 5 x_0 x_1 \\
[x_0, x_3] & = \frac{195}{2} x_0^2 -\frac{45}{2} x_0x_1 + 5 x_0 x_2\\
[x_1,x_2] & = -\frac{3}{2} x_0 x_1 + 3 x_0 x_2 + x_1^2\\
[x_1, x_3] & = 5 x_0 x_1 - 3 x_0 x_2 + 7 x_0 x_3 -\frac{5}{2} x_1^2 + x_1x_2 \\
[x_2, x_3] & = -\frac{77}{2} x_0 x_2 - \frac{77}{2} x_0 x_3+\frac{21}{2} x_1x_2+7x_1 x_3-3 x_2^2.
\end{split}
\eeq
In this paper we generalize Pym's example to arbitrary dimensions, and study the properties of the resulting algebras.
In particular, we obtain new examples of Calabi-Yau algebras in all global dimensions $\geq 5$.

Pym's example comes from an action of the two-dimensional solvable Lie algebra on $\mb{C}[x_0, x_1, x_2, x_3]$, inducing a Poisson bracket.
He shows, using a deformation formula of Coll, Gerstenhaber, and Giaquinto, that this Poisson bracket quantizes to give the algebra \eqref{pym}.
We generalize Pym's methods to construct a family of algebras $R(n,a)$, depending on a scalar $a$ and an integer $n \geq 1$.  
The $R(n,a) $ are  graded, quadratic, noetherian,  AS-regular domains of global dimension $n+1$ (Theorem~\ref{thm:ASreg}, Proposition~\ref{prop:relR}).
When $n=3$ and $a=-\frac{5}{4}$ we obtain Pym's algebra \eqref{pym}.
For another example, $R(1,a)$ is isomorphic to the {\em Jordan plane} $\mb{C}\ang{x,y}/(xy-yx-x^2)$ (if $a\neq 0$) or to the polynomial ring $\mb{C}[x,y]$ (if $a=0$).

For each $n$, there is one algebra $R(n,a)$ that is Calabi-Yau (Definition~\ref{def-CY}).  We have:
\begin{theorem}\label{ithm:CY}
{\rm (Corollary~\ref{cor:Nak2})}
The algebra $R(n,a)$ is Calabi-Yau if and only if  
\[ a = -\frac{ (n+2)(n-1)}{2(n+1)}.\]  
\end{theorem}
Example (\ref{pym}) is the algebra arising from Theorem~\ref{ithm:CY} when $n=3$.
The examples for $n \geq 4$ have not to our knowledge been studied before.

Each $R(n,a)$ induces a Poisson bracket $\{-,-\}_a$ on $\mb C[x_0, \dots, x_n]$ in the semiclassical limit.
Let $A(n,a)$ be the Poisson algebra $(\mb C[x_0, \dots, x_n], \{-,-\}_a)$. 
In such a deformation-quantisation context it is often expected that the algebra $R(n,a)$ and the Poisson algebra $A(n,a)$ share similarities. 
For instance we make a detailed study of the prime spectrum $\Spec R(n,a)$ of $R(n,a)$ and the Poisson prime spectrum $\Pspec A(n,a)$ of $A(n,a)$ and show that they are homeomorphic (Theorem~\ref{thm:spectra}).  
We further investigate the structure of this space and prove:
\begin{theorem}\label{ithm1}
{\rm (Corollary~\ref{cor:specsum})}
Let $n \geq 3$.  
The space $\Spec R(n,a) \cong \Pspec A(n,a)$ is a union of quasiprojective strata, and has dimension
\[ \begin{cases}
(n-2) & \mbox{if $a \not\in \QQ$  } \\
(n-1) & \mbox{if $a \in \QQ$.}
\end{cases}
\]
\end{theorem}
Moreover we show that $R(n,a)$ satisfies the Dixmier-Moeglin equivalence (Theorem~\ref{thm:DME}) and, using a transfer result, we prove that $A(n,a)$ satisfies the Poisson Dixmier-Moeglin equivalence (Theorem~\ref{pdme}).
Taking advantage of the fact that $\Spec R(n,a) \cong \Pspec A(n,a)$ we compute many examples of prime spectra of the $R(n,a)$. 

We investigate various ring theoretic properties of the rings $R(n,a)$ and $A(n,a)$. 
For $n \geq 2$ we show that $R(n,a) \cong R(n,b) $ if and only if $A(n,a)$ is Poisson isomorphic to $A(n,b)$ if and only if $a=b$ (Theorem~\ref{thm:isom}).
We compute the skewfield of fractions of $R(n,a)$ and show that it is isomorphic to a Weyl skewfield if and only if 
$a\in \QQ$ (Theorem~\ref{thm:skewfield}). 
We also show that the graded automorphism group of $R(n,a)$ is isomorphic to the graded Poisson automorphism group of $A(n,a)$.

Finally, we compute the point modules of the $R(n,a)$. 
We show:
\begin{theorem}\label{ithm2}
{\rm (Theorem~\ref{thm:pone})}
Let $n \geq 1$.
For any $a$, the point schemes of $R(n,a)$ are isomorphic.
The point modules of $R = R(n,a)$ are parameterised by the union $C_1 \cup \dots \cup C_{n}$, where $C_k $ is a rational normal curve of degree $k$ in $\PP(R_1^*) \cong \PP^{n}$.
\end{theorem}
As with Pym's original example, these curves correspond to various nice (equivariant with respect to the appropriate group action) ways to embed $\PP^1$ in $\Sym^{n}(\PP^1) \cong \PP^{n}$.

The organisation of the paper is as follows.  In Section~\ref{NOTATION} we define the $R(n,a)$ and $A(n,a)$, and in Section~\ref{FIRSTPROPS} we prove that $A(n,a)$ is the associated graded ring of $R(n,a)$ and that $R(n,a)$ is an Artin-Schelter regular noetherian domain.
We prove Theorem~\ref{ithm:CY} in Section~\ref{AUTGROUP}, and also calculate the graded automorphism group of $R(n,a)$ and the graded Poisson automorphism group of $A(n,a)$.
We study when two $R(n,a)$ are isomorphic in Section~\ref{ISOM} and prove Theorem~\ref{ithm1} in Section~\ref{POINTS}.

In the final three sections we describe the (Poisson) prime and primitive ideals of $R(n,a)$ and $A(n,a)$.  
We prove Theorem~\ref{ithm2} in Section~\ref{PRIMES}.  
In Section~\ref{DME} we show that $R(n,a)$ satisfies the Dixmier-Moeglin equivalence describing primitive ideals, and that $A(n,a)$ satisfies the related Poisson Dixmier-Moeglin equivalence.  We also compute the fraction skewfield of $R(n,a)$.
Finally, in Section~\ref{EXAMPLES} we give many explicit examples of prime spectra.

\smallskip
\noindent{\bf Acknowledgements.}
We thank Michel Van den Bergh, Fran\c{c}ois Dumas, Gene Freudenburg, and Diane Maclagan for helpful comments.
We particularly thank Brent Pym for extremely useful discussion on a visit of the second author to Oxford in early 2016 (including suggesting that the algebras $R(n,a)$ are all twists of each other).
This project  began when the first author was supported by an LMS mobility grant to visit the University of Edinburgh, and we thank the LMS and the University of Edinburgh for their support.

We thank the anonymous referee for their helpful comments.


\section{Notation and definitions}\label{NOTATION}

Throughout we work over a field $\kk$ of characteristic zero (in the introduction for simplicity we worked over $\mb C$).

In this section we define the algebras $R(n,a)$ and the Poisson algebras $A(n,a)$ that are the subject of the paper, and describe how they arise from $\kk^\times$-invariant actions of the two-dimensional solvable Lie algebra on polynomial rings.

We begin by discussing such actions.  
Fix an integer $n\in\ZZ_{>0}$, and let $\Delta$ be the ``downward derivation''
\[ \Delta = X_0 \del_1 + \dots + X_{n-1} \del_n\]
(where $\del_i = \del/\del X_i$).  
For $a_0, \dots, a_n \in \CC$, let $\Gamma = \Gamma(a_0, \dots, a_n)$ be the weighted Euler operator
\[ \Gamma = \sum a_i X_i \del_i.\]
We are interested in when $\Delta$ and $\Gamma$ generate a copy of the two-dimensional solvable Lie algebra inside $\Der_\kk(\kk[X_0, \dots, X_n])$.
\begin{lemma}\label{lem:one}
 We have $\Delta \Gamma - \Gamma \Delta = \Delta$ if and only if $a_j = a_0 + j$ for $j \in \{0, \dots ,n\}$.
\end{lemma}
\begin{proof}
 This follows from the computation 
\[ (\Delta \Gamma - \Gamma\Delta) X_{i+1} = (a_{i+1}-a_i) X_i,\]
so $\Delta \Gamma - \Gamma \Delta = \Delta$ if and only if $a_{i+1} = a_i +1$ for all $i$.
\end{proof}

The importance of Lemma~\ref{lem:one} is the following result of \cite{Pym}, based on the universal deformation formula of \cite{CGG}.
\begin{proposition}{\rm (\cite[Lemma~3.3.]{Pym})}
\label{prop:Pym}
Let $A$ be a commutative $\C$-algebra, and let $\Delta,\Gamma : A \to A$ be $\C$-derivations so that $\Delta$ is locally nilpotent and $[\Delta,\Gamma]=\Delta$.
For $k\in \NN$, define $\binom{\Gamma}{k}  =\frac{1}{k!} \Gamma \cdot (\Gamma-1) \cdots (\Gamma-(k-1))$.
Then
\[
f \ast g = \sum_{k=0}^\infty \hbar^k \Delta^k(f) \binom{\Gamma}{k} (g)
\]
defines an associative product on $A[\hbar]$ whose semiclassical limit as $\hbar \to 0$ is the Poisson bracket
\[\{f, g\} = \Delta(f) \Gamma(g) - \Gamma(f) \Delta(g).
\]
Further, evaluating at a particular $\hbar \in \C$ gives an associative product
\[ \ast_{\hbar}:  A \otimes_\C A \to A.\]
The rings $(A, \ast_{\hbar})$ are isomorphic for any $\hbar \neq 0$. 
\ff{We note that our sign convention in Proposition~\ref{prop:Pym} differs from that of Pym's original result.} 
\end{proposition}

Thus let   $A = \C[X_0, \dots, X_n]$.
  For $a \in \CC$, define
\[ \Gamma_a = \sum_{i =0}^n (a+i) X_i \del_i.\]
Let $\ast_a:  A \otimes_\C A \to A$ be defined by
\beq \label{mult}
f \ast_a g = \sum_{k=0}^\infty \Delta^k(f) \binom{\Gamma_a}{k} (g).
\eeq
By Lemma~\ref{lem:one} and Proposition~\ref{prop:Pym}, this defines an associative multiplication on $A$.
Let $\{ -, -\}_a:  A \otimes_\kk A \to A$ be the associated Poisson bracket
\beq\label{Poisson}
\{f, g\}_a = \Delta(f) \Gamma_a(g) - \Gamma_a(f) \Delta(g).
\eeq

Going forward, we define $R(n, a)$ to be the associative algebra $ (A, \ast_a)$ and $A(n,a) $ to be the Poisson algebra $ (A, \{-,-\}_a)$. The goal of this paper is to study these two algebras. 

To end the section, we describe how the construction above relates to the canonical action of the standard Borel subgroup of $\PGL_2(\kk)$ on $\PP^1$.
Let
\[ G = \left\{ \begin{pmatrix} a & b \\  0 & c\end{pmatrix} \st ac \neq 0 \right\} / \ \kk^\times \]
be the standard Borel subgroup of $ \PGL_2(\kk)$.
Note that $G \cong \kk \rtimes \kk^\times$ and is the subgroup of $ \PGL_2(\kk)$ that fixes the point $\infty = [1:0] \in \PP^1$.

The group $G$ acts on $\Sym^n(\PP^1) \cong \PP^n$ and thus on the homogeneous coordinate ring $B(\PP^n, \sO(1)) \cong \kk[X_0, \dots, X_n]$.
Thus the two-dimensional solvable Lie algebra $\Lie G$ acts by derivations on $\kk[X_0, \dots, X_n]$ for all $n$.
All of these actions are induced by taking symmetric powers of the standard action on $\mb V = \kk \cdot\{X, Y\}$.

To see how this works explicitly, fix $a \in \kk$ and let $\Gamma_a = a X \del_X + (a+1) Y \del_Y$ and $\Delta = X \del_Y$ as above. 
Fix also  $n \geq 1$.
If we set $X_{j} = X^{n-j} Y^j/j!$ then $X_0, \dots, X_{n}$ form a basis for $\Sym^{n}(\mb V)$ with $\Delta(X_j) = X_{j-1}$ and $\Gamma_a(X_j) = ((n-j)a +j(a+1))X_j = (na + j) X_j$.
We see that $n$'th symmetric power of the bracket $\{-,-\}_a$  on $\kk[X,Y]$ induces $\{-,-\}_{na}$ on $\kk[X_0, \dots, X_{n}]$.

Consider the  surjection $\phi:  \kk[X_0, \dots, X_n] \twoheadrightarrow \kk[X_0, X_1]^{(n)}$ induced from the Veronese embedding of $\PP^1$ as a rational normal curve in $\PP^{n}$.  
The Veronese embedding is $G$-equivariant by construction, and from the discussion above we expect $\phi$ to be a Poisson homomorphism from $A(n, a) \to A(1,a/n)^{(n)}$.  
We will see in Section~\ref{POINTS} that this does happen, and furthermore that there is also a surjection $R(n,a) \to R(1, a/n)^{(n)}$.


\section{First properties}\label{FIRSTPROPS}
Fix $n\geq 1$ and $a\in\kk$, and let $A = A(n,a)$ and $R= R(n,a)$.  
In this section, we give an explicit presentation of  $R$.  
We show that $R$ is a noetherian Artin-Schelter regular domain of global dimension $n+1$.
Finally we study the localisations of $A$ and $R$ at the (Poisson) normal element $X_0$.

 In the sequel, we will suppress the subscript $a$ where it is clear from context, so will use $\ast$ to denote multiplication in $R$ and $\{ -, - \}$ to denote the Poisson bracket on $A$.
We use concatenation to denote the (commutative) multiplication in $A$.

Let $z_j = a+j$ and set $X_{-1}:=0$. Since $\binom{\Gamma_a}{\ell}(X_j)=\binom{a+j}{\ell}X_j$ we deduce from \eqref{mult} that, for all $0\leq i,j\leq n$, we have:
\beq \label{com}
X_i\ast X_j=\sum_{\ell=0}^{i}\binom{z_j}{\ell}X_{i-\ell}X_j
\eeq
and
\beq \label{pb}
\{X_i,X_j\}=z_j X_{i-1}X_j-z_iX_{j-1}X_i.
\eeq

We will use  two gradings on $A$:  the standard {\em degree grading} $d$ defined by $d(X_i) = 1$ for all $i$ and the {\em weight grading} $\epsilon$ defined by $\epsilon(X_i) = i$.  
For $k \in \NN$, let $A_k$ be the $d$-homogeneous component of $d$-degree $k$, and let $A^k$ be the $\epsilon$-homogeneous component of $\epsilon$-degree $k$.
Let $A^j_k = A^j \cap A_k$.
Note that we have
\beq \label{epsilondelta}
\Delta(A^j_k) \subseteq A^{j-1}_k,
\eeq
whereas $\Gamma_a$ preserves both $d$-and $\epsilon$-degree.  It follows immediately 
that $R$ is also $d$-graded.  Thus the Hilbert series of $R$ (with respect to $d$) is 
\beq \label{hilbR} \hilb R = \hilb A = \frac{1}{(1-t)^{n+1}}. \eeq
Observe also that the Poisson bracket on $A$ is $d$-homogeneous:  $\{A_k, A_\ell\} \subseteq A_{k+\ell}$.  
 
On the other hand, $\epsilon$ defines a filtration on $R$:

\begin{lemma}\label{lem:assgr}
Let $A =A(n,a)$ with Poisson bracket $\{-,-\} $ and let $R = R(n,a)$, with multiplication $\ast $.  
Let $R^{\leq k} := \bigoplus_{\ell \leq k} A^\ell$, considered as a subspace of $R$. 
Then $R^{\leq k} \ast R^{\leq \ell} \subseteq R^{\leq k+\ell}$, so $R$ is filtered by the $R^{\leq k}$.  The associated graded algebra of $R$ is naturally isomorphic as a graded algebra to $A$, and under this identification we have
\[ \{ \gr f, \gr g \} =\gr(f \ast g - g \ast f).\]
\end{lemma}
\begin{proof}
As graded vector spaces, since $R=A$ certainly $\gr R = \bigoplus_k R^{\leq k}/R^{\leq k-1} \cong \bigoplus_k A^k = A$.  
Let $f = \sum_{i = 0}^k f_i, g = \sum_{i=0}^{\ell} g_i \in A$, where $f_i, g_i \in A^i$ and $f_k, g_\ell \neq 0$.  
From \eqref{mult} and \eqref{epsilondelta} we have that the $\epsilon$-degree $k+\ell$ components of both $f \ast g$ and $g \ast f$ are equal to 
$f_k g_\ell= g_{\ell} f_k $
so the multiplications on $\gr R$ and $A$ agree.

Finally,  $ \gr(f \ast g - g \ast f) $ lies in $\epsilon$-degree $k+\ell-1$ and is thus equal to 
\[  \Delta(f_k) \Gamma_a(g_\ell) - \Delta(g_\ell) \Gamma_a(f_k) =\{\gr f, \gr g\},\]
as needed.
\end{proof}

Recall that a $\kk$-algebra $R$ is {\em strongly noetherian} if $R \otimes_\kk C$ is noetherian for any commutative noetherian $\kk$-algebra $C$.  (For example, polynomial rings are strongly noetherian.)  We have:
\begin{corollary}\label{cor:noeth}
For any $n \in \ZZ_{>0}$ and $a\in\kk$, the ring $R(n,a)$ is a strongly noetherian domain.
\end{corollary}
\begin{proof}
This follows by standard arguments (see \cite[Proposition~1.6.6, Theorem~1.6.9]{MR}) from the corresponding properties for $A(n,a)  = \gr R(n,a)$.
\end{proof}

The next result is useful because it allows us to use inductive arguments to establish properties of the $R(n,a)$ or $A(n,a)$.

\begin{proposition}\label{prop:factX1} 
Suppose that $n\geq 1$. 
Let $A = A(n,a)$ and let $R= R(n,a)$.
Then
\begin{enumerate}
\item $X_0$ is normal  in $R$, and $R/ \ang{X_0} \cong R(n-1, a+1)$.
\item $X_0$ is Poisson normal in $A$, and $A/\ang{X_0} $ is Poisson isomorphic to $A(n-1,a+1)$.
\end{enumerate}
\end{proposition}

\begin{proof}

$(1)$ That $X_0 \ast R = R \ast X_0 = X_0A $  is immediate from \eqref{com}, and so $X_0$ is normal.

Let $\pi:  R \to R/\ang{X_0}$ be the canonical map, and let  $Y_i:=\pi(X_{i+1})$ for all $0 \leq i < n$. 
Clearly $R/\ang{X_0} = R/X_0 \ast R $ may be identified with $ A/X_0A = A/\ang{X_0}$ as a graded vector space, and this and $R(n-1,a+1)$ are isomorphic as graded vector spaces.
From \eqref{com}, the multiplication $\ast$ on $R/\ang{X_0}$ satisfies
\[ Y_i \ast Y_j = \sum_{\ell = 0}^{i} \binom{a+j+1}{\ell} Y_{i-\ell} Y_j,\]
which is precisely the multiplication on $R(n-1,a+1)$.

Note that the isomorphism $R/\ang{X_0} \cong R(n-1,a+1)$ respects both the $d$-grading and the $\epsilon$-filtration on $R$.

For $(2)$, it is clear that $X_0$ is Poisson normal.  
Since $A/\ang{X_0}$ is the associated graded of the $\epsilon$-filtration on $R/\ang{X_0} \cong R(n-1,a+1)$, the remaining statement follows immediately from Lemma~\ref{lem:assgr}.
\end{proof}

We next prove that the algebras $R(n,a)$ have good homological properties.
Recall that an $\NN$-graded $\CC$-algebra $R$ with $R_0 = \CC$ is {\em Artin-Schelter regular} or {\em AS-regular} if:
\begin{enumerate}
\item $ \gldim R < \infty$;
\item $R$ has finite Gelfand-Kirillov-dimension;
\item $\Ext^i_R(\CC_R, R_R) \cong \begin{cases}  0 & \text{ if $i \neq \gldim R$} \\
{}_R \CC[\ell] & \text{ if $i = \gldim R$.}
\end{cases} $.
\end{enumerate}
(Here $\CC[\ell]$ means that the module is degree-shifted by some amount $\ell \in \ZZ$.)
The Artin-Schelter regular condition is a noncommutative analogue of the good properties of commutative polynomial rings \cite{AS}.  
Condition (3) above is called the {\em AS-Gorenstein} condition.

\begin{theorem}\label{thm:ASreg}
Fix $n\in\ZZ_{>0}$ and let $a \in \kk$.
The algebra $R= R(n,a)$ is Artin-Schelter regular of  global dimension $n+1$.
\end{theorem}
\begin{proof}
We prove  by induction on $n$ that $R$ is AS-regular, Auslander-Gorenstein, and Cohen-Macaulay.  We do not give the definitions of Auslander-Gorenstein or Cohen-Macaulay; the unfamiliar reader may treat them as technical terms internal to this proof.
  
It is well-known that the Jordan plane $R(1,a)$ and the polynomial ring $R(1,0)$ satisfy all of the above properties, so that the base case $n=1$ is clear. Thus we may assume that $n > 1$.

By Lemma~\ref{lem:assgr} $A(n,a)$ is the associated graded ring of $R$ under the $\epsilon$-filtration.  Thus by \cite[Corollary~7.6.18]{MR}, $\gldim R \leq \gldim A(n,a) = n+1$.
By Proposition~\ref{prop:factX1} and  \cite[Theorem~7.3.5]{MR}, $\gldim R \geq \gldim R(n-1,a+1) +1$.
This last is  $n$ by induction.  Thus $\gldim R = n+1$.

By induction, $R(n-1,a+1)$ is Auslander-Gorenstein and Cohen-Macaulay. 
By \cite[Theorem~5.10]{Lev}, the same holds for $R$.
By \cite[Theorem~6.3]{Lev} and \eqref{hilbR}, $R$ is AS-Gorenstein and thus AS-regular.
\end{proof}

We now compute the relations for $R$.   We will use the Vandermonde identity:
 for all $a, b \in \kk$ and $k \in \NN$, we have
\beq \label{Vandermonde}
 \sum_{u=0}^k \binom{a}{u} \binom{b}{k-u} = \binom{a+b}{k}.
 \eeq
 We will also use the following lemma:
 
 \begin{lemma}\label{lem:phiinv}
 Let $V = A_1 = \Span_\kk (X_0, \dots, X_n)$.  For $b \in \kk$,
 define a linear map $\phi_b: V \to V$ by 
 \[ \phi_b(X_i) = \sum_{\ell =0}^{i} \binom{b}{\ell} X_{i-\ell}.\]
 Then  
$ \phi_a \phi_b = \phi_{a+b}$.  
\end{lemma}
\begin{proof}
By \eqref{Vandermonde} we have:
\begin{multline*}
\phi_a \phi_b(X_j)  = \phi_a \left(\sum_{i=0}^{j} \binom{b}{i} X_{j-i} \right)= \sum_{i=0}^{j} \binom{b}{i} \left( \sum_{\ell =0}^{j-i} \binom{a}{\ell} X_{j-i-\ell}\right) \\
 = \sum_{k=0}^{j} \left( \sum_{u =0}^k \binom{b}{u} \binom{a}{k-u} \right) X_{j-k} 
 = \sum_{k=0}^{j} \binom{a+b}{k} X_{j-k} = \phi_{a+b}(X_j).
 \end{multline*}  
 \end{proof}

The following lemma gives us quadratic relations that are satisfied in $R$.  We also give  an equation that  reverses the deformation formula \eqref{com} and thus allows us to obtain the commutative product from the noncommutative product $\ast$.

Recall that we set $z_j=a+j$. 

\begin{lemma}\label{lem:formulas}
For all $0\leq i,j\leq n$ we have
\begin{equation}
\label{star}
X_iX_j=X_{i}\ast X_{j}+\sum_{\ell=1}^{i}\binom{-z_j}{\ell}X_{i-\ell}\ast X_{j}=\sum_{\ell=0}^{i}\binom{-z_j}{\ell}X_{i-\ell}\ast X_{j}
\end{equation}
and
\begin{equation}
\label{rel}
X_i\ast X_j-X_j\ast X_i=\sum_{k=1}^{j}\binom{-z_i}{k}X_{j-k}\ast X_{i}-\sum_{\ell=1}^{i}\binom{-z_j}{\ell}X_{i-\ell}\ast X_{j}.
\end{equation}
\end{lemma}
\begin{proof}
Fix $j$, and recall the definition of the linear map $\phi_b$ from Lemma~\ref{lem:phiinv}.  
Since for any $F \in R_1$ we have $F\ast X_j = \phi_{z_j}(F) X_j$, clearly 
$X_iX_j = \phi_{z_j}^{-1}(X_i) \ast X_j$. 
 By Lemma~\ref{lem:phiinv}, we immediately obtain \eqref{star}.

Applying  \eqref{star} to the equation $X_i X_j = X_j X_i$, we obtain \eqref{rel}.
\end{proof}

Note that relations (\ref{rel}) can be rewritten as:
\begin{equation}
\label{relabis}
\sum_{k=0}^{j}\binom{-z_i}{k}X_{j-k}\ast X_{i}=\sum_{\ell=0}^{i}\binom{-z_j}{\ell}X_{i-\ell}\ast X_{j}.
\end{equation}
We now have:
\begin{proposition}\label{prop:relR}
The relations in $R= R(n,a)$ are exactly the $\binom{n+1}{2}$ relations given by \eqref{rel} for $0 \leq i < j \leq n$.
\end{proposition}
 \begin{proof}
 By comparing $\dim R_1 \otimes_\CC R_1$ and $\dim R_2$ it is clear that $R$ satisfies $\binom{n+1}{2} $ quadratic relations; since the relations in \eqref{rel} are linearly independent for $0 \leq i < j \leq n$ they are precisely the quadratic relations of $R$.
 
By Theorem~\ref{thm:ASreg} $R$ is AS-regular, and by \eqref{hilbR} it is what is referred to as a {\em quantum $\PP^{n}$} in \cite{SheltonTingey}.  Thus by \cite[Theorem~2.2]{SheltonTingey}, $R$ is Koszul and in particular is given by quadratic relations.  Thus the relations in \eqref{rel} are precisely the relations of $R$.
 \end{proof}
 
\begin{ex}
\label{exn=2}
 When $n=1$ we obtain:
\begin{align*}
\{X_0,X_1\}=-aX_0^2\qquad\text{and}\qquad X_0\ast X_1-X_1\ast X_0=-aX_0\ast X_0.
\end{align*}
For $a\neq0$ the algebra $R(1,a)$ is isomorphic to the well-known Jordan plane, and the Poisson algebra $A(1,a)$ is isomorphic to the \textit{Poisson-Jordan plane}: the polynomial algebra $\C[X_0,X_1]$ endowed with Poisson bracket $\{X_0,X_1\}=X_0^2$.
\end{ex}

\begin{ex}
\label{exn=3}
 When $n=2$ the algebra $R(2,a)$ is given by generators $X_0,X_1,X_2$ and relations:
\begin{align}
\label{com3}
\begin{split}
&X_0\ast X_1-X_1\ast X_0=-aX_0\ast X_0,\\
&X_0\ast X_2-X_2\ast X_0=-aX_0\ast X_1-\binom{a}{2}X_0\ast X_0,\\
&X_1\ast X_2-X_2\ast X_1=(a+2)X_0\ast X_2-(a+1)X_1\ast X_1+\binom{a+2}{2}X_0\ast X_1.
\end{split}
\end{align}
The Poisson bracket on $A(2,a)$ is given by:
\begin{align}
\label{pb3}
\begin{split}
&\{X_0,X_1\}=-aX_0^2,\\
&\{X_0,X_2\}=-aX_0X_1,\\
&\{X_1,X_2\}=(a+2)X_0X_2-(a+1)X_1^2.
\end{split}
\end{align}
\end{ex}

\begin{ex}
\label{exn=4}
When $n=3$ and $a=-5/4$ we obtain  an algebra isomorphic to Pym's example \eqref{pym}. 
The isomorphism sends $X_i\mapsto x_{i}/4^{i}$ for $i=0,\dots,3$.
\end{ex}

We now prove a technical result which gives equivalent conditions for an element to be normal (or Poisson normal), under mild conditions on $a$ and $n$. In particular we show that, up to multiplication by nonzero scalars, $X_0$ is the only homogeneous element of $d$-degree $1$ that is (Poisson) normal. 

\begin{proposition} \label{prop:pnn}
Assume either that $n\geq 2$ or that $n=1$ and $a\neq0$. Then: 

  (1a) $N$ is Poisson normal in $A = A(n,a)$ $\iff$ (1b) $N$ is normal in $R = R(n,a)$ $ \iff$ (1c) we have $\Delta(N)=0$ and $\Gamma_a(N)=uN$ for some  $u\in\C$.
  
(2a) $N$ is Poisson central in $A$ $\iff$ (2b) $N$ is central in $R$  $\iff$ (2c) we have $\Delta(N) = \Gamma_a(N) =0$.
\end{proposition}

\begin{proof}
Without loss of generality, $n > 0$.
The implications $(1c)\implies(1a),(1b)$ and $(2c) \implies (2a),(2b)$ are clear from \eqref{mult}. 
To prove the other implications, we first prove:

{\bf Claim: } 
There is an irreducible $G \in A$ with $\Delta(G) = 0$ and $\Gamma(G) = \lambda G$ for some $\lambda \neq 0$, and further such that $\ang{G} = G \ast R = R \ast G$ is a completely prime ideal of $R$.

To prove the claim, if $a \neq 0$ then we may take $G = X_0$, so $\lambda =a$. The last statement follows from the isomorphism $R/ \ang{X_0 }  \cong R(n-1, a+1)$.

 If $a = 0$ and $n \geq 2$ then let $G = X_0X_2 - \frac{1}{2} X_1^2$ and $\lambda = 2$.  We have $\Delta(G)=0$ and $\Gamma_a(G) = 2(a+1)G = \lambda G$. 
Note that $G$ is normal in $R$ and $\epsilon$-homogeneous and so the $\epsilon$-filtration on $R$ descends to $R/ \ang{G}$.  It is clear that $\gr( R/ \ang{G}) $ is isomorphic to $A/ \ang{G}$ which is a domain, so $G \ast R $ is completely prime.

We now prove  $(1a)\implies (1c)$.   Let $G, \lambda$ be as in the Claim. 
Suppose that $N$ is Poisson normal in $A$. Without loss of generality we can assume that $N\notin G A$  and $N \not\in X_0A$ since $G$ and $X_0$ are Poisson normal. 
Because $\{G,N\}=-\lambda G \Delta(N) \in NA$, there exists $G'\in A$ such that $G\Delta(N)=G'N$. Since $N\notin G A$  and $G$ is irreducible there exists $v\in A$ such that $G'=vG$ and we obtain $\Delta(N)=vN$ after dividing  by $G$.
Since $\Delta$ is $d$-homogeneous we must have $v\in A_0=\C$. But $\Delta$ is locally nilpotent so we must have $v=0$ and  $\Delta(N)=0$. 

There therefore exists $X_1'\in A$ such that:
\[X_1' N = \{X_1,N\}=X_0\Gamma_a(N).\]
Again since $N\notin  X_0 A$ we can write $X_1'=X_0u$ for some $u\in A$, and we have $\Gamma_a(N)=uN$. The derivation $\Gamma_a$ being $d$-homogeneous, we conclude as before that $u\in A_0=\C$ is a scalar.

We now prove $(1b)\implies(1c)$. Suppose that $N$ is normal in $R$. Without loss of generality we can assume that $N\notin X_0 \ast R$ and $N \not \in G \ast R$ since $X_0$ and $G$ are normal. 
By normality of $N$ there exists $G'$ such that $ G \ast N = N \ast G' $. 
Since $N\notin G \ast R$ we must have $G'\in G \ast R$ as $G \ast R$ is completely prime.
Moreover, using the equation $G \ast N = N \ast G'$ we deduce that $G'$ is $d$-homogeneous and $d(G') = d(G)$. So $G'=wG$ for a scalar $w$, and we conclude that $N$ must satisfy the equation $G \ast N = w N \ast G$. Using the definition of the product $\ast$ we obtain the following equation after dividing by $G$:
\begin{equation}
\label{eq22}
w \left( N+\lambda\Delta(N)+\frac{\lambda(\lambda-1)}{2}\Delta^2(N)+\cdots+\binom{\lambda}{k} \Delta^{k}(N) \right) = N
\end{equation}
for some $k\geq0$. We decompose $N=\sum_{i=0}^\ell N_i\in\bigoplus A^i$ into $\epsilon$-homogeneous  pieces. Since $\Delta(A^i)\subseteq A^{i-1}$ we deduce from equation \eqref{eq22} that $N_\ell=wN_\ell$, i.e. $w=1$ and $\Delta(N_\ell)=0$.                                                                                                                                                                                                                                                                                                                                                                                                                                                                    Finally we obtain by (a decreasing) induction that $\Delta(N_i)=0$ for all $i$, so that $N\in\ker\Delta$.

	By normality of $N$ there exists $F\in R$ such that $X_1\ast N= N \ast F$. 
Since $\Delta(N)=0$ we have
\[ NF = N \ast F = X_1 \ast N = X_1 N + X_0 \Gamma_a(N),\]
so $N$ divides $X_0 \Gamma_a(N) \in A$.  
As $N \not \in X_0\ast  R = X_0 A$ we have  that $N$ divides $\Gamma_a(N)$, so there exists $u \in A$ with $\Gamma_a(N) = uN$.
Comparing $d$-degrees we have $u \in A_0 = \CC$.

For the remaining equivalences note that if $N$ is either Poisson central in $A$ or central in $R$ then certainly $N$ is (Poisson) normal.  We have seen that  $\Delta(N) = 0$ and $\Gamma_a(N) = uN$ for some $u \in \CC$.  We then have either:
$ 0 = \{ X_1, N\} = uX_0 N $ (if $N$ is Poisson central) or $0 = X_1  \ast N - N \ast X_1 = u X_0 N$ (if $N$ is central).  In either case $u=0$.  
Thus (2a) $\implies$ (2c) and (2b) $ \implies $ (2c).
\end{proof}

\begin{remark} \label{rem:deleta} We note that Proposition~\ref{prop:pnn} does not make it easy to {\em find} the normal (or central) elements of $R$.  In particular, it is a famously difficult problem in symbolic dynamics to calculate $A^{\Delta} = \{ f \in A \st \Delta(f) = 0\}$; see \cite{Freudenburg}.  In fact, $A^{\Delta}$ is not known explicitly for $n \geq 9$.
\end{remark}

The rings $R(n,a)$ are complicated to study but they become much simpler after localisation at the normal element $X_0$. To end this section we study the localisations of $R=R(n,a)$ and $A=A(n,a)$ at the (Poisson) normal element $X_0$. Let $R^\circ := R[X_0^{-1}]$ and $A^\circ := A[X_0^{-1}]$. We note that $\Delta$ and $\Gamma_a$ extend to $A^\circ$, and that the Poisson bracket $\{-,-\}_a$ extends to $A^\circ$ and is still defined by \eqref{Poisson}.  
Likewise, the multiplication $\ast_a$ on $R^\circ$ is still defined by \eqref{mult}.
The $d$-grading extends to $A^\circ$ and to $R^\circ$.
The $\epsilon$-grading extends to $A^\circ$, and $\epsilon$ defines a filtration on $R^\circ$ with associated graded $A^\circ$, as with the non-localised rings.  

We further define a new grading  $A^\circ = \bigoplus_{e\in \kk} A^\circ(e)$ by defining $A^\circ(e)$ to be the $e$-eigenspace of $\Gamma_a$.  Notice that $A^\circ(e) \neq  0$ if and only if  $e  \in a \ZZ + \NN$.   Further note that if $F \in (A^\circ)^\epsilon_d$, then $F\in A^\circ(e)$ for $e = d a+\epsilon$.
It follows from \eqref{epsilondelta} that $A^\circ$ is also the associated graded ring of $R^\circ$ with respect to the $e$-filtration, and that the $\epsilon$- and $e$-filtrations on $R^\circ$ both induce the  Poisson bracket $\{-,-\}_a$ on $A^\circ$.

Let $B = (A^\circ)^\Delta$.   In contrast to $A^\Delta$, which is unknown in general, $B$ is easy to compute and is isomorphic to a polynomial ring in $n$ variables.
Explicitly, for $j \geq 1$ define 
\[ Y_j = \sum_{p = 0}^\infty \frac{(-X_1)^p}{p! X_0^p } \Delta^p(X_j) .\]
This is well-defined since $\Delta$ is  locally nilpotent.
Note that $Y_j \in (A^\circ)^{j}_1$; in particular, $Y_j$ is $e$-homogeneous, with $e(Y_j) = e(X_j) = z_j$.

\begin{lemma}\label{lem:four}
 We have that $A^\circ $ is freely generated by $X_1, Y_2, \dots, Y_n$ and $B$ is freely generated by $Y_2, \dots, Y_n$  as  commutative algebras over $\kk[X_0^{\pm 1}]$.
\end{lemma}
\begin{proof}
 That $B$ is generated by $X_0^{\pm 1}, Y_2, \dots, Y_n$  follows from \cite[Proposition 2.1]{vdEssen}, since $\Delta(X_1/X_0) = 1$.
If $i \geq 2$, then 
\beq\label{boff} X_i - Y_i \in \kk[X_0^{\pm 1} ,X_2, \dots, X_{i-1}]. \eeq
 This shows that the elements $X_0, X_1, Y_2, \dots, Y_n$ are algebraically independent, proving that $B$ is freely generated as claimed.

By \eqref{boff} and induction,  $A^\circ $ is generated over $B$ by $X_1$, and thus over $\kk[X_0^{\pm 1}]$ by $X_1, Y_2, \dots, Y_n$.    
Since $\Kdim A^\circ = n+1$, thus $A^\circ$ is freely generated as well.
\end{proof}

\begin{lemma} \label{lem:fourprime}  Let $\delta = X_0 \Gamma_a$, so $\delta \in \Der_\kk(B)$.
Define a map $\beta:  B[z] \to A^\circ$ by $\beta(\sum b_i z^i) = \sum b_i X_1^i$.
Then $\beta$ is a ring isomorphism $B[z; \delta] \to R^\circ$ and a Poisson isomorphism $(B[z],\{-. -\}_\delta) \to A^\circ$.  
\end{lemma}
\begin{proof}
Since  $B$ is an $e$-graded subring of $A^\circ$, therefore $\delta(B) \subseteq B$ and so $\delta \in \Der_\kk(B)$.

Let $\star$ be the multiplication in $B[z; \delta]$.
Note that  if we put $e(z) = a+1$ then $\beta$ is clearly $e$-graded as an isomorphism of vector spaces and $e$-filtered as a map from $B[z; \delta] \to R^\circ$.
Since $A^\circ$ is the associated graded of $R^\circ$ it suffices to prove that $\beta:B[z;\delta] \to R^\circ$ is a ring homomorphism.   
It is enough to check that $\beta(z \star g) = \beta(z) \ast_a g$ for $g \in B(e)$.  
But we have
\[ \beta(z) \ast_a g = X_1 \ast_a g = gX_1 + egX_0   = \beta(gz + \delta(g)) = \beta(z \star g)\]
as needed.
\end{proof}


\section{Graded automorphisms and the Nakayama automorphism} \label{AUTGROUP}

In this section we compute the graded automorphism group of $R(n,a)$ and determine its Nakayama automorphism. 
In particular we  prove Theorem~\ref{ithm:CY}.

We begin with the graded automorphism group of $R(n,a)$.
In order to prove that the maps $\phi_b$ defined in \ref{lem:phiinv} are well-defined automorphisms of $R(n,a)$ for any $a,b\in\kk$ we use the theory of Zhang twists. More specifically we prove that $R(n,a)$ and $R(n,b)$ are Zhang twists of each other for any $a,b\in\kk$. We now recall the definition of Zhang twist from \cite{Zhang}. If $S$ is an $\NN$-graded ring with multiplication $*$, and $\phi\in \Aut(S)$ is a graded automorphism of $S$, then the {\em Zhang twist} of $S$ by $\phi$ is written $S^{\phi}$.  As a graded vector space, $S^{\phi}$ is isomorphic to $S$.  The multiplication $\circ$ on $S^\phi$ is defined by $r \circ s= r* \phi^i(s)$ for all $r\in S^{\phi}_i = S_i$ and $s\in S^\phi_j = S_j$.  This is associative by \cite[Proposition~2.3]{Zhang}. 

In the terminology of \cite{Zhang} the family of maps $\{\phi^i\ |\ i\in\N\}$ is a twisting system of $S$.  We note that the definition in \cite{Zhang} is slightly more general than the one we give here, but we do not need this greater generality.

Also recall that we may define a {\em left-hand Zhang twist} of $S$ by $\phi$, which we write ${}^{\phi}S$.  The multiplication  $\circ'$ on ${}^{\phi}S$ is defined by $r \circ' s = \phi^j(r)* s$ for all $r \in {}^{\phi}S_i = S_i$ and $s\in {}^{\phi}S_j = S_j$.  This is associative by \cite[Proposition~4.2]{Zhang}. We have the following easy lemma.
\begin{lemma}\label{lem:LRtwist}
Let $S$ be an $\NN$-graded ring with multiplication $*$, and let $\phi \in \Aut(S)$ be a graded automorphism of $S$.  Then $S^{\phi} \cong {}^{\phi^{-1}}S$.
\end{lemma}
\begin{proof}
Define $\Phi:  S^\phi \to {}^{\phi^{-1}} S$ by $\Phi(t) = \phi^{-k}(t)$ for all $t \in S^{\phi}_k$.

We denote the multiplication on $S^\phi$ by $\circ$ and the multiplication on ${}^{\phi^{-1}}S$ by $\circ'$.  
Then for $r \in S_i, s\in S_j$,  we have
\[ \Phi(r \circ s) = \Phi(r * \phi^i(s)) = \phi^{-(i+j)}(r * \phi^i(s)) = \phi^{-i-j}(r) * \phi^{-j}(s)\]
since $\phi$ is an algebra automorphism of $S$.
Thus 
\[ \Phi(r) \circ' \Phi(s) =\phi^{-j}( \phi^{-i}(r)) * \phi^{-j}(s) = \Phi(r \circ s),\]
as needed.
\end{proof}

The first goal of this section is to prove the following theorem
\begin{theorem}\label{thm:ONE}
Fix $n \in \ZZ_{\geq 0}$.  For any $a,b \in \CC$, the rings $R(n,a)$ and $R(n,b)$ are right and left Zhang twists of each other.
\end{theorem}

We will prove this by constructing some explicit graded automorphisms of $R(n,a)$.

\begin{lemma}\label{lem:TWO}
For any $a \in \CC$, the map 
\[\phi_a(X_j) = \sum_{i = 0}^{j} \binom{a}{i} X_{j-i}\]
induces a $d$-graded automorphism  of $R(n,a)$, which we also denote $\phi_a$.
\end{lemma}

\begin{proof}
Note that 
\[ X_j * X_0 = \phi_a(X_j) X_0= X_0 \phi_a(X_j) = X_0 * \phi_a(X_j).\]
Thus $\phi_a$ is simply the automorphism of $R(n,a)$ induced by conjugating by the normal element $X_0$.
 \end{proof}

\begin{proposition}\label{prop:THREE}
For any $a \in \CC$, let $\phi_a$ be the automorphism of $R(n,a)$ defined in Lemma~\ref{lem:TWO}.  
Then  $R(n,0) \cong R(n,a)^{\phi_a}$, under the map induced by sending $X_i \mapsto X_i$.
\end{proposition}

\begin{proof}
By Lemma~\ref{lem:LRtwist}, and Lemma~\ref{lem:phiinv}, it suffices to prove that $R(n,0) \cong {}^{\phi_{-a} }R(n,a)$.
Let $\circ$ denote the multiplication on ${}^{\phi_{-a} }R(n,a)$.
Then for any $i$ and $j$, we have
\begin{align*}
X_i \circ X_j = \phi_{-a}(X_i) * X_j & = \sum_{\ell =0}^{i} \binom{-a}{\ell} X_{i-\ell} * X_j \\
 &  =  \sum_{\ell =0}^{i} \sum_{k=0}^{j-\ell } \binom{-a}{\ell} \binom{a+j}{k} X_{i-\ell - k} X_j \\
 & = \sum_{u =0}^{i} \sum_{v=0}^u \binom{-a}{v} \binom{a+j}{u-v} X_{i-u} X_j \\
 & = \sum_{u=0}^{i} \binom{j}{u} X_{i-u} X_j,
 \end{align*}
 where we have used the Chu-Vandermonde identity at the end.  
 This agrees with the formula for multiplication in $R(n,0)$.
 \end{proof}
 
 As a corollary of Proposition~\ref{prop:THREE} we obtain that $\phi_b$ induces an automorphism of $R(n,a)$ for any $a,b$.  For we have
 \begin{lemma}\label{lem:FOUR}
 Let $S$ be an $\NN$-graded ring and let $\phi$, $\psi$ be  graded automorphisms of $S$ such that $\phi \psi = \psi \phi$.
 Then $\psi$ is also an automorphism of $S^{\phi}$ and ${}^{\phi}S$.
 \end{lemma}
 \begin{proof}
 We prove the lemma for $S^{\phi}$; note the statement makes sense because as a graded vector space $S = S^{\phi}$.
 Let $*$ denote the multiplication on $S$ and let $\circ$ denote the multiplication on $S^\phi$.
 We check that for any $r \in S_i$ and $s \in S_j$, we have
\[
 \psi (r \circ s) = \psi (r * \phi^i(s)) = \psi(r) * \psi \phi^i(s) = \psi(r) \ast \phi^i \psi(s) = \psi (r) \circ \psi(s),\]
as needed.
\end{proof}

 \begin{corollary}\label{cor:automs}
 Fix $n\in\ZZ_{>0}$.  For any $a,b\in \kk$, the action of $\phi_b$ on $R(n,a)$ induces an automorphism of $R(n,a)$.
 \end{corollary}
 \begin{proof}
 That $\phi_b$ induces an automorphism of $R(n,0)$ for any $b$ follows from Lemma~\ref{lem:FOUR}, with $\phi = \psi = \phi_b$.
 We then apply Lemma~\ref{lem:FOUR} again, using the facts that $\phi_a \phi_b = \phi_{a+b} = \phi_b \phi_a$ and that $R(n,a) \cong {}^{\phi_a} R(n,0)$ proved in Proposition~\ref{prop:THREE}.
 \end{proof}
 
 \begin{corollary}\label{cor:FIVE}
 For any $a,b \in \kk$, we have ${}^{\phi_b} R(n,a) \cong R(n, a+b)$, under the map $X_i \mapsto X_i$.
 \end{corollary}
 \begin{proof}
 The proof is very similar to the proof of Proposition~\ref{prop:THREE} and is left to the reader.
 \end{proof}
 
 \begin{proof}[Proof of Theorem~\ref{thm:ONE}]
The result follows immediately from Corollary~\ref{cor:FIVE}, using Lemma~\ref{lem:LRtwist} to move from left to right twists.
 \end{proof}

We now show that up to composition with nonzero scalar multiplication, the only non trivial graded automorphisms of $R(n,a)$ are the maps $\phi_b$. We will need the following lemma.

\begin{lemma}\label{lem:identity}
Let $n \geq 2$. Suppose that $\psi \in \Aut R(n,a)$ is such that $\psi(X_0) = X_0$ and $\psi(X_j) = X_j + \alpha_j X_0$ with $\al_j\in\kk$ for $1 \leq j \leq n$.
Then $\psi $ is the identity.
\end{lemma}

\begin{proof}
By applying $\psi$ to the second equation of \eqref{com3} we obtain
\[[X_0, X_2] = [X_0,X_2+\al_2X_0]=\psi([X_0,X_2])=[X_0,X_2]-a\al_1 X_0 \ast X_0.\]
Thus we have $\al_1=0$ when $a\neq0$. If $a=0$ we apply $\psi$ to the third equation of \eqref{com3} to get
\begin{multline*}
(X_1 + \alpha_1 X_0) \ast (X_2+\alpha_2 X_0) - (X_2 + \alpha_2 X_0) \ast (X_1 + \alpha_1 X_0)= [X_1,X_2]=\\
[X_1,X_2]-2\alpha_1 X_0 \ast X_1 + (2\alpha_2-\alpha_1^2+\alpha_1) X_0 \ast X_0.
 \end{multline*}
Again we must have $\alpha_1=0$. Now suppose that $\alpha_j = 0$ for $1 \leq j < k$. We apply the automorphism $\psi$ to the relation (\ref{rel}) with $i=1$ and $j=k$. After rearranging we get 
\begin{equation}
\label{1k}
X_1 \ast (X_k + \alpha_k X_0) - (a+k) X_0 \ast (X_k + \alpha_k X_0) = \\
(X_k + \alpha_k X_0) \ast X_1 + \sum_{j = 1}^{k} \binom{-a-1}{j} X_{k-j} \ast X_1.
\end{equation}
Thanks to equation (\ref{com}) we compare the coefficients of the $X_0^2$ term in both sides of the equality (\ref{1k}) and we obtain $ \alpha_k(a-(a+k)) = 0$. So $- k\alpha_k=0$ and $\alpha_k = 0$ since $k\neq0$.
\end{proof}

We can now determine the graded automorphism group $\Aut_{\gr}R(n,a)$ of $R(n,a)$.
For $\lambda \in \kk^\times$, let $\xi_\lambda$ be the automorphism that scales all $X_i$ by $\lambda$.

\begin{theorem}
\label{prop:gradaut}
Assume either that $n\geq 2$ or that $n=1$ and $a\neq0$. The $d$-graded automorphisms of $R(n,a)$ are 
of the form $\xi_\lambda \phi_c$, for some $\lambda \in \kk^\times$ and $c \in \kk$.
 In particular we have $\Aut_{{\gr}}R(n,a)\cong \kk^{\times} \times \kk$. 
\end{theorem}

\begin{proof}
It follows from Corollary \ref{cor:automs} that the maps $\xi_\lambda \phi_c$ are indeed automorphisms of $R(n,a)$. 
Reciprocally, we proceed by induction on $n$. 
We first assume that $a\neq-(n-1)$. For $n=1$ (and $a\neq0$), the result follows from \cite[Theorem 3.1]{Shirikov}. 
We assume that the result is true for $R(n-1,b)$ with $b\neq-(n-2)$. 
Let $\psi$ be a $d$-graded automorphism of $R(n,a)$. 
Thanks to Proposition \ref{prop:pnn} the only normal elements of $R(n,a)$ with $d$-degree $1$ are the nonzero scalar multiples of $X_0$. 
Thus, up to composition with some $\xi_\lambda$, we have that $\psi(X_0)=X_0$. 
In particular $\psi$ induces a $d$-graded automorphism of $R(n,a)/\langle X_0\rangle\cong R(n-1,a+1)$. By the induction hypothesis we have $\psi(X_j)=\sum_{i=0}^{j-1}\lambda\binom{c}{i}X_{j-i}+\alpha_{j}X_0$ for any $j\geq 1$, where $\lambda\in\kk^{\times}$ and $c,\alpha_j\in\kk$. We obtain that $\lambda=1$ by applying $\phi$ to the first equation of \eqref{com3} when $a\neq0$, or to the third equation of \eqref{com3} when $a=0$. In particular we have $\psi=\phi_c$ modulo $X_0$, and by applying Lemma \ref{lem:identity} to $\psi\circ\phi_{-c}$ we conclude that $\psi=\phi_c$.


We finally deal with the case $a=-(n-1)$. 
We prove by induction on $n\geq2$ that the $d$-graded automorphisms of $R(n,-(n-1))$ are of the form $\xi_\lambda \phi_c$.
 We only prove the base case of the induction since the induction step of the previous induction will also  apply to that case.
Let $\psi$ be a $d$-graded automorphism of $R(2,-1)$. 
Then $\psi(X_0)=\lambda X_0$  for some nonzero $\lambda$, and by rescaling we may assume that $\lambda=1$.  
Further, $\psi$ induces a $d$-graded automorphism of the commutative polynomial ring $R(2,-1)/\langle X_0\rangle\cong R(1,0)\cong\kk[X_1,X_2]$. Therefore there exist $c,d,\alpha,\beta,\gamma,\delta$ such that $\psi(X_1)=\alpha X_2+\beta X_1 +c X_0$ and $\psi(X_2)=\gamma X_2+\delta X_1 +d X_0$, where $\alpha\delta-\beta\gamma\neq0$.

Applying $\psi$ to the relations \eqref{com3} we obtain that $\alpha=0$, $\beta=\gamma=1$, $\delta=c$ and $d=\binom{c}{2}$. This shows that $\psi=\phi_c$ and concludes the proof.
\end{proof}

One can prove a similar version of Proposition \ref{prop:gradaut} for the Poisson algebra $A(n,a)$. For any $c\in\kk$ the maps $\varphi_c:=\exp(c\Delta)$ are well-defined automorphisms of the commutative polynomial ring $A(n,a)$ since $\Delta$ is a locally nilpotent derivation. Moreover we observe that
\[\Delta(\{X_i,X_j\})=\{\Delta(X_i),X_j\}+\{X_i,\Delta(X_j)\}\]
for all $0\leq i,j\leq n$. This shows that $\Delta$ is a Poisson derivation of $A(n,a)$, and thus the maps $\varphi_c$ are Poisson automorphisms of $A(n,a)$. 

We state the following result without proof.
\begin{proposition}
\label{prop:pgradaut}
Assume either that $n\geq 2$ or that $a\neq0$. The $d$-graded Poisson automorphisms of $A(n,a)$ are 
of the form $\xi_\lambda \varphi_c$, for some $\lambda \in \kk^\times$ and $c \in \kk$.
 In particular the graded Poisson automorphism group $\PAut_{{\rm gr}}A(n,a)$ of $A(n,a)$ is isomorphic to $\kk^\times \times \kk$.
\end{proposition}

\begin{corollary}
Assume either that $n\geq 2$ or that $a\neq0$. Then 
\[\PAut_{{\rm gr}}A(n,a)\cong \kk^{\times}\times \kk \cong\Aut_{{\rm gr}}R(n,a).\]
When $n=1$ and $a=0$ we have
\[\PAut_{{\rm gr}}A(1,0)=\Aut_{{\rm gr}}R(1,0)=\Aut_{{\rm gr}}\kk[X,Y]\cong\GL_2(\kk).\]
\end{corollary}

\begin{rem}
We note that for certain values of $a$, there exist non-graded (Poisson) automorphisms. For instance it is well known that for any polynomial $P\in\kk[X_0]$ the map sending $X_0\mapsto X_0$ and $X_1\mapsto X_1+P$ defines an automorphism of the Jordan plane $R(1,a)$ and the Poisson-Jordan plane $A(1,a)$. When $n=2$ and $a=1/q$ for some $q\in\Z_{>0}$, the map $f$ defined by
\[f(X_0)=X_0,\quad f(X_1)=X_1+X_0^{q+1}, \quad f(X_2)=X_2+X_0^qX_1+X_0^{2q+1}\]
is a non-graded Poisson automorphism of $A(2,1/q)$.
\end{rem}

To conclude this section we calculate the Nakayama automorphism of $R(n,a)$ and prove, in particular, that $R(n, -\frac{1}{2} (n+2)(n-1)/(n+1))$
is $(n+1)$-Calabi-Yau for every $n > 0$. 
We first recall some definitions. Let $R$ be an a $\kk$-algebra, and let $R^e = R\otimes_\kk R^{op}$ be the {\em enveloping algebra} of $R$.
An $(R,R)$-bimodule $M$ can be considered as a left $R^e$-module by defining $r\otimes s \cdot m = rms$.

\begin{defn}\label{def-CY}
We say that $R$ is {\em skew Calabi-Yau } (or {\em skew CY}) if
\begin{itemize}
\item[(i)] $R$ is {\em homologically smooth}:  $R$ has a finite projective resolution as a left $R^e$-module such that each term is finitely generated;
\item[(ii)]  There are an algebra automorphism $\mu$ of $R$ and an integer $d$ such that
\[ 
\Ext^i_{R^e}(R, R^e) \cong \begin{cases} 0 & \text{if $i \neq 0$} \\
 {}^1 R^\mu & \text{ if $i=d$.}
 \end{cases}
 \]
 (Here ${}^1 R^\mu$ is the $R$-bimodule which is isomorphic to $R$ as a $\kk$-vector space and such that $r\cdot s \cdot t = r s \mu(t)$.)
\end{itemize} 
 If $R$ is skew CY, the automorphism $\mu$ is called the {\em Nakayama automorphism} of $R$.  If $\mu$ is inner, then $R$ is {\em Calabi-Yau} or {\em CY}.
 \end{defn}
 
 By \cite[Lemma~1.2]{RRZ}, any AS-regular connected graded algebra is skew CY. In particular, the algebras $R(n,a)$ are skew CY.
 
We will need the following lemma before calculating the Nakayama automorphism of $R(n,a)$.
\begin{lemma} \label{lem:Nak0}
Let $\mu$ be the Nakayama automorphism of $R= R(n,a)$ for any $n \geq 1$.
Then for all $0 \leq j \leq n$, we have $\mu(X_j) - X_j \in \Span(X_0, \dots, X_{j-1})$.
In particular $\mu(X_0) = X_0$.
\end{lemma}
\begin{proof}
The Nakayama automorphism of the commutative polynomial ring $A = \kk[X_0, \dots, X_n]$ is well-known to be trivial.
Since $A$ is the associated graded of $R$ with respect to the $\epsilon$-filtration and clearly $A^e \cong \kk[X_0, \dots, X_{2n+1}]$ is the associated graded of $R^e$, by \cite[Proposition~3.1]{Bjork}, we have that the associated graded of $\Ext^{n}_{R^e}(R, R^e) = {}^1 R^{\mu}$ is a subquotient of $\Ext^{n}_{A^e}(A, A^e) = {}^1 A^1$.  
This shows that the $\epsilon$-leading term of $\mu(X_j) $ must be $X_j$.
\end{proof}

\begin{theorem}\label{thm:Nak1}
For any $n \geq 1$, the Nakayama automorphism of $R(n,a)$ is $\phi_c$, where 
\[c= (n+1)a + \binom{n+1}{2}-1.\]
\end{theorem}
\begin{proof}
We begin by calculating the Nakayama automorphism of $R(n,0)$.  
For $n=1$ we have $R(1,0) \cong \kk[X_0,X_1]$, so the Nakayama automorphism is the identity $\phi_0$; we have $0 = \binom{2}{2}-1$.
Suppose now that $n>1$ and we wish to calculate the Nakayama automorphism  $\mu$ of $R(n, 0)$.
Let $b= n+ \binom{n}{2}-1 = \binom{n+1}{2}-1$.
Since by Lemma~\ref{lem:Nak0} we have $\mu(X_0) = X_0$, thus $\mu $ induces an automorphism of $R(n, 0) / \ang{X_0 }  \cong R(n-1, 1)$.
By \cite[Lemma~1.5]{RRZ}, using the fact that $X_0$ is central, this induced automorphism is equal to the Nakayama automorphism of $R(n-1,1)$, which by induction on $n$ is $\phi_b$.
Thus modulo $X_0$, we have that $\mu =   \phi_b$.  
By applying Lemma~\ref{lem:identity} to $\psi = \mu \circ \phi_{-b}$ we see that $\mu = \phi_b$.

Let $\mu$ be the Nakayama automorphism of $R(n,0)$ and $\nu$ be the Nakayama automorphism of $R(n,a) = {}^{\phi_a} R(n,0)$.
By \cite[Theorem~0.3]{RRZ}, we have $\nu = \mu \phi_a^{n+1} \xi_\lambda$ for some $\lambda \in \kk^\times$.
By Lemma~\ref{lem:Nak0}, we must have $\lambda=1$.
Thus
\[ \nu = \mu \phi_a^{n+1} = \phi_{\binom{n+1}{2}-1 + (n+1)a},\]
as claimed.
\end{proof}

\begin{corollary}\label{cor:Nak2}
For any $n \geq 1$, 
the algebra $R(n,a)$ is Calabi-Yau if and only if  
\[ a = \frac{1-\binom{n+1}{2}}{n+1} = -\frac{ (n+2)(n-1)}{2(n+1)}.\]  
\end{corollary}

\begin{example}\label{eg:pym}
When $n=3$, we have $-\frac{ (n+2)(n-1)}{2(n+1) }= -\frac{5}{4}$, so Pym's example~\ref{pym} is Calabi-Yau, using Example~\ref{exn=4}.
\end{example}


\section{Isomorphisms}\label{ISOM}
We have seen in Example \ref{exn=2} that we have $R(1,0) \cong R(1,a)$ $\iff$ $a=0$ and that $R(1,a) \cong R(1,b)$ for any nonzero $a,b \in \CC$.
A similar statement holds for the Poisson algebras $A(1,a)$.
In this section we analyse the isomorphism question for $R(n,a)$ and $A(n,a)$ where $n \geq 2$. Moreover we show that each $ R(n,a)$ is isomorphic to its opposite ring.

The main theorem of this section is the next result.  
\begin{theorem}\label{thm:isom}
Let $n \geq 2 $ and let $a, a' \in \CC$.
The following are equivalent:
\begin{enumerate}
\item $a=a'$;
\item $R(n,a) \cong R(n, a')$;
\item $A(n,a)$ is Poisson isomorphic to $A(n,a')$.
\end{enumerate}
\end{theorem}

We remark that Theorem~\ref{thm:isom} is not particularly surprising given Corollary~\ref{cor:Nak2}, since for each $n$ there is a unique $a$ with $R(n,a)$ Calabi-Yau.

\begin{proof}[Proof of Theorem~\ref{thm:isom}]
 $(1) \implies (2), (3)$ is trivial.
 
 $(3) \implies (1)$.  Let $A = A(n,a)$ and let $A' = A(n,a')$.
 To avoid any confusion, we will denote the generators of $A$ by $X_0, \dots, X_n$ as usual and denote the corresponding generators of $A'$ by $X_0', \dots, X_n'$.
We will let $\Delta$ also denote the downward derivation on $A'$, so $\Delta(X_i') = X_{i-1}'$.
 
 Suppose that there is a Poisson isomorphism $\alpha:  A \to A'$.
 For $0 \leq i \leq n$ set $ \alpha(X_i)= T_i = L_i + P_i $, where $L_i \in A'_1$ and $P_i \in \CC \oplus \bigoplus_{k \geq 2} A'_k$.
 	 Note that since $\alpha$ can be seen as an algebra automorphism of the polynomial ring $A$, the linear parts $L_i$ must be nonzero for all $0\leq i\leq n$ (using the fact that the Jacobian of $\alpha$ must be nonzero at the origin).
	 
If $a=0$ then since $X_0$ is Poisson central $T_0$ must be as well.
As the Poisson bracket respects the $d$-grading on $A'$, the degree 1 part $L_0$ must be Poisson central and by Proposition~\ref{prop:pnn} we have $L_0 \in \ker \Delta$.  
Thus $L_0 = \lambda_0 X_0'$ is Poisson central and so $a' =0$.  This shows that $a=0$ $\iff$ $a'=0$.

Suppose now that $aa' \neq 0$.
Now $T_0$ is Poisson normal and by Proposition~\ref{prop:pnn} we have $T_0 \in \ker \Delta$ and $T_0$ is a $\Gamma_{a'}$-eigenvector.  Thus $L_0 \in \ker \Delta$ and so $L_0 = \lambda_0 X_0'$ for some $0 \neq \lambda_0 \in \CC$.
Since $\Gamma_{a'}(L_0) = a' L_0$ we have $\Gamma_{a'}(T_0) = a' T_0$.

Let $c = a/a'\neq0$.
Applying $\alpha$ to the equation $ a \Delta(g) X_0 = \{ g, X_0\}_a$ we obtain
\[ 
a \alpha \Delta (g) T_0 = \{ \alpha(g), T_0\}_{a'} = a' \Delta \alpha(g) T_0\]
and so
\beq\label{eqboop}
\Delta \alpha = c \alpha \Delta
\eeq
as maps from $A \to A'$.
Since $\Delta$ respects the $d$-degree we have $\Delta(L_i ) = c L_{i-1}$ for $ 1 \leq i \leq n$.
Thus there are $\lambda_1, \lambda_2 \in \CC$ so that
\begin{align*}
L_0&= \la_0 X_0',\\
L_1&= c \la_0 X_1' + \la_1 X_0',\\
L_2&= c^2 \la_0 X_2' + c \la_1 X_1' + \la_2 X_0'.
\end{align*}

From \eqref{eqboop} we also conclude that $T_0, T_1 \in \im \Delta$ and so have no constant term. Considering the equality
\beq \label{popeye}
 (a+2) T_0 T_2 - (a+1) T_1^2 = \alpha( \{X_1,X_2 \}_a) = \{ T_1, T_2 \}_{a'}
 \eeq
then gives that $T_2$ has no constant term. Therefore the $X_0' X_2'$ term of the left-hand side of \eqref{popeye} is $(a+2) c^2 \lambda_1^2 X_0' X_2'$. On the other hand, the $X_0' X_2'$ term of the right-hand side is 
$ \Delta(c \lambda_0 X_1') \Gamma_{a'}(c^2 \lambda_0 X_2')= c^3 \lambda_0^2 (a'+2) X_0' X_2'$.
Thus 
$\frac{a+2}{a} = \frac{a'+2}{a'}$ and $a' = a$.

$(2) \implies (1)$.  
Let $R = R(n,a)$ and let $R' = R(n,a')$. 
As above, we denote the generators of $R'$ by $X_0', \dots, X_n'$.
If $R \cong R'$ then  by \cite[Theorem 0.1]{BZ} there exists a graded isomorphism $\alpha: R \to R'$.
Since $X_0$ is normal in $R_1$ its image $\alpha(X_0)$ is normal in $R_1'$ and $\alpha(X_0)\in\Span(X_0')$ thanks to Proposition~\ref{prop:pnn}.
We prove the result by induction on $n$. 

We first consider the case $n=2$. 
Since $X_0$ is central if and only if $a=0$ we deduce that $a = 0 \iff a' = 0$. Therefore we may assume that $a,a'\neq0$ in the following.
Since $\alpha(X_1)\in R_1'$ there exist scalars $\mu_0,\mu_1$ and $\mu_2$ not all zero such that $\alpha(X_1)=\mu_0X_0'+\mu_1X_1'+\mu_2X_2'$. We fix $\lambda_0\in\kk^{\times}$ such that $\alpha(X_0)=\lambda_0X_0'$. 
By applying $\alpha$ to the first equality in (\ref{com3}) we obtain
\[\lambda_0\mu_1(X_0'\ast X_1'-X_1'\ast X_0')+\lambda_0\mu_2(X_0'\ast X_2'-X_2'\ast X_0')=-a\lambda_0^2 X_0'\ast X_0'.\]
After simplification using the relations (\ref{com3}) we compare the coefficients of the $X_0'\ast X_1'$ term.
We get $-a'\lambda_0\mu_2=0$ which implies that $\mu_2=0$ since $a'\neq0$. 
In particular we have $\alpha(X_1) \in \Span(X_0', X_1')$ and $\alpha$ respects the $\epsilon$-filtrations on $R$ and $R'$.
By Lemma~\ref{lem:assgr} $\alpha$ induces a Poisson isomorphism between $A(2,a)$ and $A(2, a')$, so $a=a'$.  

Suppose now $n\geq3$ and the result true for $n-1$. Since $\alpha(X_0)\in\Span(X_0')$ there is an isomorphism $R(n-1,a+1) \cong R/\ang{X_0} \to R'/\ang{X_0' } \cong R(n-1,a'+1)$ and we have $a+1 = a'+1$.
%
%
%
\end{proof}

To end the section we prove that $R(n,a)$ and $R(n,a)^{\rm{op}}$ are isomorphic. This will be used in the next section.

\begin{theorem}
\label{antiiso}
For any $u\in\kk$, the map $\omega_u$ from $R(n,a)$ to $R(n,a)$ defined by 
\begin{equation*}
\label{antiisom}
\omega_u(X_i)=(-1)^{i}\sum_{\ell=0}^{i}\binom{i-u}{\ell}X_{i-\ell} = (-1)^{i} \phi_{i-u}(X_i)
\end{equation*}
is an anti-isomorphism. In particular $R(n,a)$ and $R(n,a)^{\rm{op}}$ are isomorphic for any $a\in\kk$ and any $n>0$.
\end{theorem}

We will need the following lemmas.
 
\begin{lemma}
\label{binomeq1}
For any $v\in\mathbb{N}$ and $a,m\in\kk$ we have
\begin{equation}
\label{eq11}
\sum_{k=0}^v (-1)^k\binom{a}{k}\binom{m-k}{v-k}=\binom{m-a}{v}.
\end{equation}
\end{lemma}
\begin{proof}
The proof is a  combination of the Chu-Vandermonde identity (\ref{Vandermonde}) and the identity
\begin{equation*}
\label{id11}
\binom{u}{\ell}=(-1)^{\ell}\binom{\ell-u-1}{\ell}
\end{equation*}
for any $\ell\in\mathbb{N}$ and $u\in\kk$.
\end{proof}


\begin{lemma}
For $a,u\in\kk$ we have
\begin{equation}
\label{soleil}
\omega_u \phi_a = \omega_{u+a}.
\end{equation}
\end{lemma}

\begin{proof}
We have
\begin{align*}
\omega_u\phi_a(X_i)
&=\sum_{k=0}^{i}\binom{a}{k}\sum_{\ell=0}^{i-k}(-1)^{i-k}\binom{i-k-u}{\ell}X_{i-k-\ell}\\
&=(-1)^{i}\sum_{k=0}^{i}\sum_{\ell=0}^{i-k}(-1)^{-k}\binom{a}{k}\binom{i-k-u}{\ell}X_{i-k-\ell}\\
&=(-1)^{i}\sum_{v=0}^{i}\sum_{k=0}^{v}(-1)^{-k}\binom{a}{k}\binom{i-k-u}{v-k}X_{i-v}\\
&=(-1)^{i}\sum_{v=0}^{i}\binom{i-(u+a)}{v}X_{i-v}=w_{a+u}(X_i),
\end{align*} 
where we have set $v=k+\ell$ and then use Lemma \ref{binomeq1} with $m=i-u$. 
\end{proof}

\begin{proof}[Proof of Theorem~\ref{antiiso}]
Note that the relations (\ref{relabis}) in $R(n,a)$ can be rewritten as
\begin{equation}
\label{rela1}
\phi_{-z_i}(X_j)\ast X_i=\phi_{-z_j}(X_i)\ast X_j.
\end{equation}
for any $0\leq i,j\leq n$.
To see that $\omega_u$ is a well-defined anti-automorphism of $R(n,a)$ it is enough to apply $\omega_u$ to both side of (\ref{rela1}) and check that we obtain the same result. Using (\ref{soleil}) we have
\begin{align*}
\omega_u(\phi_{-z_i}(X_j)\ast X_i)&=\omega_u(X_i)\ast \omega_u\phi_{-z_i}(X_j)\\
&=\omega_u(X_i)\ast \omega_{u-z_i}(X_j)=(-1)^{i+j}\phi_{i-u}(X_i)\ast\phi_{j-u+z_i}(X_j)
\end{align*}
and similarly
\begin{align*}
\omega_u(\phi_{-z_j}(X_i)\ast X_j)=(-1)^{i+j}\phi_{j-u}(X_j)\ast\phi_{i-u+z_j}(X_j).
\end{align*}
For any $c\in\kk$ the map $\phi_c$ is an automorphism of $R(n,a)$, and by applying $\phi_c$ to (\ref{rela1}) we obtain that 
\[\phi_{c-z_i}(X_j)\ast \phi_{c}(X_i)=\phi_{c-z_j}(X_i)\ast \phi_{c}(X_j).\]
Letting $c=a+i+j-u=i-u+z_j=j-u+z_i$ we deduce that
\[\phi_{j-u}(X_j)\ast \phi_{i-u+z_j}(X_i)=\phi_{i-u}(X_i)\ast \phi_{j-u+z_i}(X_j).\]
This proves that $\omega_u(\phi_{-z_i}(X_j)\ast X_i)=\omega_u(\phi_{-z_j}(X_i)\ast X_j)$ and finishes the proof.
\end{proof}

 The opposite algebra $A(n,a)^{\rm{op}}$ of the Poisson algebra $A(n,a)$ is the same associative $\kk$-algebra endowed with the opposite Poisson bracket $\{-,-\}^{\rm{op}}:=-\{-,-\}$.
The analogue  of Theorem \ref{antiiso} for $A(n,a)$ follows easily from Lemma~\ref{lem:assgr}. It is straightforward to check that the map $\beta$ from $A(n,a)^{\rm{op}}$ to $A(n,a)$ sending $X_i$ to $(-1)^iX_i$ for all $i$ is a Poisson algebra isomorphism.





\section{Point modules} \label{POINTS}

Let $R$ be an $\NN$-graded ring.  A  (left or right) {\em point module} over $R$ is a cyclic graded module $M$ with $\operatorname{hilb}(M) = 1/(1-t)$.
In this section we study the point modules of the $R(n,a)$ and prove:

\begin{theorem}\label{thm:pone}
Let $R = R(n,a)$, where $n \geq 1$.  
If $M = R/J$ is a right or left point module, then $J$ is generated by $J_1$ as a right (respectively, left) ideal of $R$.
There is a projective scheme $X$ that parameterises both left and right point modules, and $X^{\rm red}$ is isomorphic to $n$ copies of $\PP^1$, where the $k$'th copy is embedded in $V(X_0, \dots, X_{n-k-1}) \subseteq\PP(R_1^*)$ as a rational normal curve of degree $k$.
\end{theorem}

Recall that there are normal elements $Y_2, \dots, Y_n$ in $R^\circ = R[X_0^{-1}]$.  Let $T $ be the subalgebra of $R^\circ$ generated by $X_0, X_1, Y_2, \dots, Y_n$.  
We will compute the point modules for $R$ by relating points of $R$ to those of $T$.

Let $C = \kk[X_0, Y_2, \dots, Y_n]$.   We know from Lemma~\ref{lem:fourprime} that $R^\circ = C[X_0^{-1}][X_1; \delta]$ where $\delta$ is the derivation $X_0 \Gamma_a$ of $C[X_0^{-1}]$.  The proof of that result also gives that $T = C[X_1; \delta]$.  We then have:
\begin{proposition}\label{prop:Tone}
\begin{enumerate}
\item $T$ has $\binom{n+1}{2}$ quadratic relations.  They are:
\[ \begin{array}{ccc} 
A: \quad & X_1X_0 - X_0 X_1 - a X_0^2 & \\ 
 A(k):  & X_1 Y_k -  Y_k X_1- (a+k) X_0 Y_k  \quad & \mbox{ for $2 \leq k \leq n$}\\
 B(k):  &  X_0 Y_k - Y_k X_0 & \mbox{ for $2 \leq k \leq n$}\\
 C(j,k): & Y_k Y_j - Y_j Y_k & \mbox{ for $2 \leq j < k \leq n$.}
 \end{array} \]
\item $T$ is Artin-Schelter regular.
\end{enumerate}
\qed
\end{proposition}
\begin{proof}
$(1)$ is a straightforward computation.  For $(2)$, note that the weight grading $\epsilon$ can be used to define a filtration on $T$ whose associated graded ring is the commutative polynomial ring $C[X_1]$.  Thus  $T$ has finite global dimension.    Clearly $T$ has finite Gelfand-Kirillov dimension.  Note that $Y_2, \dots, Y_n$ are normal in $T$, and $T/(Y_2, \dots, Y_n)$ is isomorphic to the Jordan plane and is AS-regular. Thus $T$ is AS-Gorenstein by \cite[Corollary~5.10, Theorem~6.3]{Lev}.   
\end{proof}

By Corollary~\ref{cor:noeth} $R$ is strongly noetherian.  Thus by \cite[Theorem~E4.3]{AZ} there is a projective scheme $X$ that parameterises right $R$-point modules up to isomorphism.
For $x \in X$, let $M(x)$ be the corresponding point module.
By \cite[Proposition~10.2]{KRS}, there is  $\sigma \in \Aut(X)$ so that $M(x)[1]_{\geq 0} \cong M(\sigma(x))$ for all $x\in X$.

We also have:
\begin{lemma}\label{lem:ptwo}
If $M$ is an $R$-point module or a $T$-point module, then either $MX_0 =0$ or $M$ is $X_0$-torsionfree.
\end{lemma}
\begin{proof}
This is a  consequence of the facts that $R$ and $T$ are AS-regular with Hilbert series $1/(1-t)^{n+1}$ and $X_0$  is a normal element.
See \cite[p. 728]{LBS} for a summary of the argument.
\end{proof}

We now relate $X_0$-torsionfree point modules over $R$ to point modules over $T$.  

\begin{proposition}\label{prop:pthree}
Let $M$ be an $X_0$-torsionfree point module over $R$.  Then there is a unique $T$-action on $M$ that extends the action of $T \cap R$ and makes $M$ a $T$-point module.  

Similarly, if $L$ is an $X_0$-torsionfree point module over $T$ then there is a unique $R$-action on $L$ that extends the action of $T\cap R$ and makes $L$  also a point module over $R$.
\end{proposition}
\begin{proof}
Let $x\in X$ and recall the isomorphism  $M(x)[1]_{\geq 0} \cong M(\sigma(x))$.
Equivalently, there are inclusions $M(x) \subseteq M(\sigma^{-1}(x))[1]$ for all $x\in X$.  
Note also that because $\sigma \in \Aut(X)$, the point $\sigma^{-1}(x)$ is the unique $y \in X$ so that $M(x) \subseteq  M(y)[1]$.  
Let 
\[
N(x) := \varprojlim M(\sigma^{-n}(x))[n].
\]
The module $N(x)$ is $\ZZ$-graded with $\dim N(x)_k = 1$ for all $k \in \ZZ$.
(In fact, $N(x)$ is the injective hull of $M(x)$ in the category of graded $R$-modules.)

Suppose now that $M(x)$ is $X_0$-torsionfree.  By Lemma~\ref{lem:ptwo} so is each $M(\sigma^{-n}(x))$, as $M(x) \subseteq M(\sigma^{-n}(x))[n]$.
Thus $N(x)$ is $X_0$-torsionfree.
We may thus choose a basis $\{n_k \st k \in \ZZ\}$ of $N(x)$ with $n_k \in N(x)_k$ and $n_k X_0 = n_{k+1}$.   
If we define $n_k X_0^{-1} = n_{k-1}$ for all $k \in \ZZ$, we  obtain an action of $R[X_0^{-1}]$ on $N(x)$, and thus an action of $T$ on $N(x)_{\geq 0} = M(x)$.
The action is clearly unique.

The proof that an $X_0$-torsionfree point module over $T$ has an induced $R$-action is similar.

If $K$ is an $X_0$-torsionfree point module over $T$ (or over $R$)  then $K$ is cyclic as a module over $ \kk[X_0]$.  Thus $K$ is also cyclic under the induced $R$-action (or $T$-action).
\end{proof}

We now compute the point modules of $T$, at least up to radical.  
Let  $V:= T/\ang{Y_2, \dots, Y_n}$. We have $V  \cong k \ang{X_0,X_1}/\ang{X_1X_0-X_0X_1-aX_0^2}$.  If $a \neq 0$, this ring is the 
 Jordan plane; in particular, the $V$ are isomorphic for any $a\neq 0$.  If $a=0$, then $V$ is a commutative polynomial ring.

\begin{proposition}\label{prop:pfour}
The reduced point scheme of $T$ is isomorphic to  $V(X_0) \cup V(Y_2, \dots, Y_n) \subseteq \PP(T_1^*)$.
In particular, an $X_0$-torsionfree point module over $T$ must be annihilated by $Y_2, \dots, Y_n$.

The point modules over $T$ that are annihilated by $Y_2, \dots, Y_n$ are parameterised by $\PP^1$.
\end{proposition}
\begin{proof}
The final statement comes directly from the isomorphism of $V$ with the Jordan plane (if $a \neq 0$) or with $\kk[X_0, X_1]$ (if $a=0$).

We multilinearise the relations of $T$ as in \cite{ATV}, to compute the scheme $X(2) \subseteq \PP(T_1^*)^{\times 2}$ parameterising {\em truncated} point modules over $T$ with Hilbert series $1+t+t^2$.   
Let the coordinates on $\PP(T_1^*)^{\times 2}$ be $X_0, X_1, Y_2, \dots, Y_n, X_0', X_1', Y_2', \dots, Y_n'$.
We thus obtain an $\binom{n+1}{2}\times (n+1)$ matrix $A$ with entries in $T_1$ so that $X(2)$ is defined by the equations
\[ A \begin{bmatrix} X_0' \\ X_1' \\ Y_2' \\ \vdots \\ Y_n' \end{bmatrix} = 0.\]
The rows of $A$ are given by:
\[A: \quad \begin{bmatrix}   X_1-aX_0 & -X_0 & 0 \dots \end{bmatrix} \]
\[ A(k): \quad \begin{bmatrix}   0 & -Y_k  & \dots & X_1 - (a+k) X_0 & \dots  \\
					& & & (k+1) & \end{bmatrix}\]
\[ B(k):  \quad \begin{bmatrix}   -Y_k & 0 & \dots & X_0 & \dots  \\
					& & & (k+1) & \end{bmatrix}\]					
\[ C(j,k):  \quad \begin{bmatrix} 0 & \dots & Y_k & \dots &-Y_j & \dots \\				
                                                        	& & (j+1) & & (k+1) \end{bmatrix}.\]
(Here we use $(\ell)$ to indicate the column of an entry.)	

Let $X'$ be the projection of $X(2)$ onto the first coordinate. It is standard that $X'$ is defined by the locus where $\rank(A) < n+1$.    
Consider the minor $A_k$ of $A$ given by rows $A, A(k),$ and $B(2)-B(n)$.  In columns $j \geq 2, j \neq k$ the only nonzero entry of $A_k$ is the $(j+1,j+1)$ entry $X_0$.  Thus
\[ \det A_k = \pm X_0^{n-2} \begin{vmatrix} X_1 -a X_0 & -X_0 & 0 \\
   								0 & -Y_k & X_1 - (a+k)X_0 \\
								-Y_k & 0 & X_0
								\end{vmatrix}
								= \pm k X_0^{n}   Y_k .\]
Since $k\neq 0$ we have that $X'$ is a closed subscheme of $ V(X_1^{n} Y_2, \dots, X_1^{n} Y_n) $.

Let $Y$ be the reduced point scheme of $T$.  Since $T$ is strongly noetherian, for some $N$ we have $Y \subseteq \PP(T_1^*)^{\times N}$.  Let $X$ be the projection of $Y$ to the first factor; we have $X \subseteq X'$.  
To prove that this projection induces an isomorphism $Y \cong V(X_0) \cup V(Y_2, \dots Y_n)$ it suffices to prove that each point of $V(X_0) \cup V(Y_2, \dots Y_n)$ corresponds to a point module.  
That is, if $W$ is a codimension 1 subspace of $T_1$ with either $X_0 \in W$ or $Y_2, \dots Y_n \in W$ we must show that $T/WT$ is a point module.
If $X_0 \in W$ then $T/WT$ is isomorphic to the right module over $T/X_0T$ defined by factoring out the image of $W$.  Since $T/X_0T $ is commutative any codimension 1 subspace of $(T/X_0 T)_1$ defines a point module, so $T/WT $ is a point module over $T$.

Now suppose that $Y_2, \dots, Y_n \in W$.  Then $T/WT$ is isomorphic to the right module over $V$ given by factoring out the image of $W$.  As $V$ is  isomorphic either to the Jordan plane or to $\kk[X_0, X_1]$, we see likewise that any codimension 1 subspace of $T_1$ that contains $Y_2, \dots, Y_n$ defines a point module.
\end{proof}

The natural map $T\to V$ induces a graded homomorphism $\zeta:  T[X_0^{-1}] \to D$, where $D = Q_{gr}(V) = V[ v^{-1} : v \in V \mbox{ is nonzero and homogeneous}]$.
Let $u = X_1 X_0^{-1}$ and let $t = X_0$.  It is well-known that $D \cong \kk(u)[t, t^{-1}; \sigma]$, where $\sigma(u) = u-a$.  To check this we verify that we have $X_1 X_0 = ut^2 = t(u+a)t = X_0(X_1+aX_0)$.  

Our next result computes $\zeta(R) \subseteq \kk(u)[t; \sigma]$.  To prove Theorem~\ref{thm:pone}, the case $a=0$ is the only one needed, but we give the general result because it is of independent interest.

\begin{proposition}\label{prop:pfive}
We have  $\zeta(R)\cong V^{(n)}$, the $n$-th Veronese of $V$.
\end{proposition}
\begin{proof}
Recall that $R=R(n,a)$, where $n \geq 1$.  Note also that the subring of $R $ generated by $X_0, \dots, X_k$ is isomorphic to $R(k,a)$ for all $1 \leq k \leq n$.  We will abuse notation and write $R(1,a)\subseteq R(2,a) \subseteq \dots \subseteq R(n,a)$, and will use this to  prove the theorem by induction on $n$.  Likewise, we write $A(1,a) \subseteq A(n,a)$.  

We will show, for all $k$, that $\zeta(X_k) \in k[u]\cdot t$ and, more specifically, that 
\beq \label{phixk}
\zeta(X_k)t^{-1} \in \kk[u]  \mbox{  and has leading term $ \frac{1}{k!} u^{k} $.}
\eeq

We need a subsidiary lemma.  Recall that the underlying space of $R $ is the commutative ring $A$, with multiplication indicated by  $\ast$  in $R$ and juxtaposition  in $A$.  
When we write an expression like $X_1^j$, we mean the commutative power; we will write $X_1^{\ast j}$ to mean the noncommutative power.

\begin{lemma}\label{lem:psix}
For $0 \leq k \leq n$, the following hold.
\begin{enumerate}
\item For any $f \in R(1,a)$, we have
\[ X_k \ast f - X_k f \in \sum_{i=0}^{k-1} X_i \ast R(1,a) = \sum_{i=0}^{k-1} X_i A(1,a).\]
\item
\[ \sum_{i=0}^k X_i \ast R(1,a) = \sum_{i=0}^k X_i A(1,a).\]
\end{enumerate}
\end{lemma}

\begin{proof}  Again, we induct on $k$; the result is trivial for $k=0$.  Assume we know the result for $k$.

$(1)$.   We may assume that $f$ is $e$-homogeneous.  We have 
\[ X_{k+1} \ast f - X_{k+1}  f = \sum_{i=1}^{k+1} X_{k+1-i} \binom{\Gamma}{i} (f).  \]
Each  $X_{k+1-i} \binom{\Gamma}{i} (f)$ is a scalar multiple of $ X_{k+1-i} f$.  Thus the right hand side is in $\sum_{i=0}^k X_{i} A(1,a) = \sum_{i=0}^k X_i \ast R(1,a)$ by induction.

$(2)$ follows directly from $(1)$.
\end{proof}

Since $\zeta(X_0), \zeta(X_1) \in \kk[u][t; \sigma]$ we obtain that $\zeta(R(1,a))\subseteq \kk[u][t; \sigma]$.

Since $\zeta$ is graded, we have $\zeta(X_1^j) = f_j(u) t^{j}$.  
We have $\zeta(X_1^{\ast j}) = (ut)^j  = u(u-a)(u-2a)\cdots(u-(j-1)a) t^j$, which has leading term $u^jt^j$.
We immediately obtain from Lemma~\ref{lem:psix}(1) (and an elementary induction) that 
\beq\label{friday1}
f_j = u^j + \mbox{ (lower order terms)} \in \kk[u].
\eeq

Assume now that we have shown that \eqref{phixk} holds for $0, \dots, k-1$, and 
consider the image of $X_0^{k} X_k = X_0^{\ast (k)} \ast X_k$ under $\zeta$.   
Recall that
\[ Y_k = X_k + \sum_{p=1}^{k} \frac{(-1)^p X_{k-p} X_1^p}{p! X_0^p}.\]
Thus
\beq\label{friday2}
\zeta(X_0^{k} X_k) = \zeta\left(X_0^{k} Y_k + \sum_{p=1}^{k} \frac{(-1)^{p+1} X_0^{k-p} X_{k-p} X_1^p}{p! }\right).
\eeq
We know that $\zeta(X_0^{k}Y_k) = 0$ since $\zeta(Y_k) = 0$.
Applying \eqref{phixk}, Lemma~\ref{lem:psix}(1), and \eqref{friday1}, we obtain that
\[ \zeta(X_0^{k-p} X_{k-p} X_1^p) = \frac{1}{(k-p)!} \bigl( u^{k} + \mbox{lower terms} \bigr) t^{k+1}.\]
Thus \eqref{friday2} reduces to 
\[ \zeta(X_0^{k} X_k) = \left(  \Big(\sum_{p=1}^{k}  \frac{(-1)^{p+1}}{p! (k-p)!} \Big) u^{k} + \mbox{lower terms} \right) t^{k+1}.\]
Since commuting with $t$ does not effect the  leading term of a polynomial in $u$, it follows that
\[  \zeta(X_k) = t^{-k} \zeta(X_0^{k} X_k) = \left(  \Big(\sum_{p=1}^{k}  \frac{(-1)^{p+1}}{p! (k-p)!} \Big) u^k + \mbox{lower terms} \right) t,\]
and that $\zeta(X_k) t^{-1} \in \kk[u]$.
Finally, 
\[ \sum_{p=1}^{k}  \frac{(-1)^{p+1}}{p! (k-p)!}  = 
\sum_{p=0}^{k}  \frac{(-1)^{p+1}}{p! (k-p)!} + \frac{1}{k!} =  \frac{1}{k!},\]
as needed.  We have thus established \eqref{phixk} for $X_k$.

It follows from \eqref{phixk} that 
\beq \label{phiR}
\zeta\big(R(k,a)\big)_1 = \kk\cdot\{t, ut, \dots, u^{k}t\}
\eeq
 for all $k$.

To complete the proof of the proposition, we must identify the ring $S = \zeta(R) = \kk \ang{t, ut, \dots, u^{n}t} \subseteq \kk(u)[t; \sigma_a]$.
Denoting the point $[1:0] \in \PP^1$ by $\infty$, we may identify $S_1 $ with global sections of the sheaf $ \sO_{\PP^1}(n\infty)$, multiplied by $t$.  
It is then clear that $S$ is isomorphic to the $n$th Veronese of the {\em twisted homogeneous coordinate ring} (see \cite{AV}) $B(\PP^1, \sO(1), \sigma_{a/n})$.
We have that $ B(\PP^1, \sO(1), \sigma_{a/n})$ is isomorphic to the Jordan plane if $a \neq 0$ and to $\kk[X_0, X_1]$ if $a=0$ by \cite[Example~4.13]{Rogalski}.
Thus, no matter the value of $a$ we have $ S \cong V^{(n)}$.
\end{proof}

Note that we have constructed above the homomorphism $R(n,a) \to R(1,a/n)^{(n)}$ that was predicted at the end of Section~\ref{NOTATION}.
Taking associated graded rings, we obtain also the predicted Poisson homomorphism $A(n,a) \to A(1,a/n)^{(n)}$.
Further, if we let $1 \leq k \leq n$, then factoring out $X_0, \dots, X_{n-k}$ and applying this construction to the factor we obtain a surjection $R(n,a) \to R(1, (a+n-k)/k)^{(k)}$ and thus a $\PP^1$ of point modules corresponding to point modules of $R(1, (a+n-k)/k)$.
These modules are all predicted by the discussion in Section~\ref{NOTATION}.
In a sense the striking content of Theorem~\ref{thm:pone} is that there are no other point modules for $R(n,a)$.

We now prove Theorem~\ref{thm:pone}.

\begin{proof}[Proof of Theorem~\ref{thm:pone}]
By Theorem~\ref{antiiso}, it suffices to prove the result for right point modules.  
We prove the result by induction on $n$, first noting the result is trivial for $n=1$, since point modules over $R = R(1,a)$ are in bijection with codimension-1 subspaces of $R_1$ for any $a$.  Thus we may suppose that $n>1$.
Since $R(n,a)$ is a Zhang twist of $R(n,0)$ (Theorem \ref{thm:ONE}), by \cite{Zhang} it suffices to prove the result for $R = R(n,0)$.

Let $M = R/J$ be a point module, and recall the definition of the homomorphism $\zeta$ from just before Proposition~\ref{prop:pfive}.  We claim that $J = J_1 R$.
  If $X_0 \not \in J$ then by Lemma~\ref{lem:ptwo} $M$ is $X_0$-torsionfree and by Proposition~\ref{prop:pthree}, $M$ is also an $X_0$-torsionfree $T$-point module.
Thus $MY_k = 0$ for $k \in \{2, \dots, n\}$ by Proposition~\ref{prop:Tone}, so $M$ is annihilated by $\ker \zeta$.  
By Proposition~\ref{prop:pfive},  $R/\ker \zeta $ is isomorphic to $\kk[X,Y]^{(n)}$. The claim follows from the fact that if $J$ is a right ideal of $\kk[X,Y]^{(n)}$ such that $(\kk[X,Y]^{(n)})/J$ is a point module, then $J$ is  generated in degree one.

We may thus view the point scheme $X$ of $R$ as contained in $\PP(R_1^*)$, where $M = R/J$ corresponds to the codimension 1 subspace $J_1 \subseteq R_1$.
By slight abuse of notation, we regard a point module $M$ as  a closed point of the point scheme $X \subseteq \PP(R_1^*)$.
Let $Z(1)$ be the Zariski closure of $\{M = R/J \ | \ X_0 \not \in J\}$.
Then $\ker \zeta \subseteq J$ for all $M = R/J$ in $Z(1)$.
Thus $Z(1)$ is isomorphic to the point scheme of $R/\ker \zeta \cong \kk[X,Y]^{(n)}$, which is the image of the degree $n$ Veronese map from $\PP^1 \to \PP^{n} $; in other words, a degree $n$ rational normal curve in $\PP(\kk[X,Y]_{n}^{*}) = \PP(R_1^*)$.

It remains to consider the case of point modules annihilated by $X_0$.  By induction (using Proposition~\ref{prop:factX1} as usual), these form a bouquet of $n-1$ rational normal curves in $V(X_0)$, and the result follows.
\end{proof}

\begin{remark}\label{rem:pseven}
Let $a,b \in\kk$, and let $\phi_b$ be the automorphism of $R(n,a)$ defined in Lemma~\ref{lem:TWO}. 
It is easy to see that $\phi_b$ extends to an automorphism of $R(n,a)[X_0^{-1}] =:R^\circ(n,a)$.
By Corollary~\ref{cor:FIVE}, we have ${}^{\phi_b}R^\circ(n,a) \cong R^\circ(n, a+b)$.
Let $\ast_a$ denote multiplication in $R(n,a)$.
Since $[X_0, Y_k] = 0$ in $R^\circ(n, a+b)$, we have
\[ \phi_b(Y_k) \ast_a X_0 = \phi_b(X_0) \ast_a Y_k = X_0 \ast_a Y_k = Y_k \ast_a X_0,\]
so $\phi_b(Y_k) = Y_k$.
It follows that $\ker \zeta$ is $\phi_b$-invariant.
One can use this to give an alternate proof of the $X_0$-torsionfree case of Theorem~\ref{thm:pone}.
\end{remark}


\section{Primes of $R$ and Poisson primes of $A$}\label{PRIMES}

The results of the previous section construct surjections from $R(n,a)$ to $R(1,a')^{(n-k)}$ for all $0 \leq k < n$ (where $a'$ depends on $a$, $n$, and $k$).
The kernels of these maps are of course prime ideals, and are in some sense independent of $a$:  for example, the kernel of the map $R(n,a) \to R(1,a/n)^{(n)}$ is generated by $Y_2, \dots, Y_n$ no matter the value of $a$.
(This is a slight abuse of notation, since $Y_2, \dots, Y_n$ are in the localisation $R(n,a)[X_0^{-1}]$.)
We will see that, similar to the above, the $d$-graded prime spectrum of $R(n,a)$ is largely independent of $a$.
On the other hand, the ungraded primes of $R(n,a)$ depends very sensitively on $a$.

First, though, we explore the connection between primes of $R(n,a)$ and Poisson primes of $A(n,a)$.
It is well-known that there is often a close relationship between prime ideals of a noncommutative ring $R$ and Poisson primes of its semiclassical limit.
Fix $n$ and $ a$ and let $R = R(n,a)$ with multiplication $\ast_a$.
Let $A =A(n,a)$ with Poisson bracket $\{-,-\}_a$.  
We will show that $\Spec R = \Pspec A$ in the strongest possible sense.
That is, we prove:
\begin{theorem}\label{thm:spectra}
Let $P \subseteq A = R$.  
Then: 
\begin{enumerate}
\item $P \in \Spec R$ if and only if $P \in \Pspec A$, and further every prime ideal of $R$ is completely prime;
\item $P$ is a primitive ideal of $R$ if and only if $P$ is a Poisson primitive ideal of $A$.
\end{enumerate}
\end{theorem}
In the statement of Theorem~\ref{thm:spectra}, recall that an ideal $I$ of a ring $R$ is called \emph{left (resp. right) primitive} if it is the annihilator of a simple left (resp. right) $R$-module. Thanks to Theorem~\ref{antiiso} the two notions coincide for the ring $R(n,a)$ and we don't specify left or right for a primitive ideal in this article. 
An ideal $P$ of a Poisson ideal $A$ is called \emph{Poisson primitive} if it is the largest Poisson ideal contained in a maximal ideal of the commutative ring $A$.

We further give a stratification of $\Spec R = \Pspec A$, and show the strata are homeomorphic to commutative (projective, projective-over-affine, or affine) varieties.
We compute also the $d$-homogeneous primes of $R$, and show that (if $a\not\in \ZZ$) they do not depend on the precise value of $a$.

 David A. Jordan's work \cite{Jordan} on Poisson algebras and Ore extensions is crucial to this section.
The following is a slight strengthening of \cite[Theorem~3.6]{Jordan}.
\begin{proposition}\label{prop:jordan}
Let $B$ be a noetherian $\CC$-algebra that is a domain and let $\delta$ be a nonzero derivation of $B$.
Let $R = B[z; \delta]$, which we write as the left $B$-module $R = \bigoplus_{n \geq 0} B z^n$, with multiplication $\ast$ such that 
\beq\label{ghana}
z \ast b = b \ast z + \delta(b) \quad \mbox{ for $b \in B$}.
\eeq
Let $A$ be the Poisson-Ore extension $(B[z], \{-,-\}_\delta)$ (see \cite[Lemma~3.1]{Jordan}); as a ring $A = B[z]$, with Poisson bracket defined by 
\beq\label{PO}
 \{a z^m, bz^n\}_\delta = (ma \delta(b) - nb \delta(a))z^{m+n-1}.
 \eeq
Let $P \subseteq R = A$.  Then: 
$P \in \Spec R$ if and only if $P \in \Pspec A$. 
Further, if $P \in \Spec R = \Pspec A$, then either $P \supseteq \delta(B)$ or $P $ is generated by the $\delta$-invariant prime ideal $P \cap B$ of $B$.
\end{proposition}

For the proof of Proposition~\ref{prop:jordan}, recall that a {\em $\delta$-ideal} $I$ of $B$ is an ideal $I$ with $\delta(I)\subset I$.  
The $\delta$-ideal $I$ is {\em $\delta$-prime} if for all $\delta$-ideals $J, K$ of $B$, we have that $JK \subseteq I$ implies that $J\subseteq I $ or $K \subseteq I$.  Since $\operatorname{char} \CC = 0$, a $\delta$-prime ideal of $B$ is prime by \cite[Lemma~1.1]{Goodearl}.

\begin{proof}[Proof of Proposition~\ref{prop:jordan}]
This proof is largely a recapitulation of the proof of \cite[Theorem~3.6]{Jordan}, pointing out that the homeomorphism constructed there is in fact the identity map.
Let $J = \delta(B) B$.  For any $\delta$-ideal $Q$ of $B$, we have $Q \ast R = R \ast Q= QA = Q[z]$ (using the identification of $R$ with $ A$ as left $B$-modules).
In particular, this holds for $Q=J$.

Let $P \subseteq A = R$ and suppose that  either $P \in \Spec R$ or $P \in \Pspec A$.
If $J \subseteq P$  then $P \supseteq JA = J \ast R= R \ast J$.  
From \eqref{ghana}, the two multiplications on the graded vector space $A/JA = R/J \ast R$ are equal, and  the induced Poisson bracket on $A/JA$ is trivial.
Let $C = R/J \ast R = A/JA$.  
Then $J \subseteq P \in \Spec R$ holds if and only if $P/J\ast R = P/JA \in \Spec C = \Pspec C$, if and only if $J \subseteq P \in \Pspec A$.

Now suppose that $J \not \subseteq P$ and let $Q = P \cap B$; note $Q$ is  a $\delta$-prime ideal of $B$.  By \cite[Lemma~3.2]{Jordan}, if $P \in \Pspec A$ then $P = QA = Q \ast R$, and thus $P \in \Spec R$.
But by \cite[Lemma~3.3]{Jordan}, if $P \in \Spec R$ then $P = Q \ast R = QA$ and $P \in \Pspec A$.

Finally, we note that the results in \cite{Jordan} are stated for $\kk = \mathbb{C}$, but are valid over any field of characteristic 0.
\end{proof}

We now prove Theorem~\ref{thm:spectra}. Recall from Section \ref{FIRSTPROPS} that $R^{\circ}=R[X_0^{-1}]$ and $A^{\circ}=A[X_0^{-1}]$.

\begin{proof}[Proof of Theorem~\ref{thm:spectra}]
$(1)$.
If $n=1$ the result is well-known. So let $n>1$ and let $P\subseteq A = R$.

First suppose that $X_0 \not \in P$.
Let $P^{\circ}:=P[X_0^{-1}] = P + X_0^{-1} P + \dots \subseteq A^\circ = R^\circ$.
If $P \in \Spec R$ or $P \in \Pspec A$ then $P = P^{\circ} \cap A$.
So it suffices to prove that $P^{\circ} \in \Spec R^\circ $ $ \iff$ $P^{\circ} \in \Pspec A^\circ $.
This is an immediate application of Lemma~\ref{lem:fourprime} and Proposition~\ref{prop:jordan}.
That prime ideals of $R^\circ$ are completely prime is shown in  \cite{Sigurdsson} (see also \cite[Remark~3.7]{Jordan}) and the rest of $(1)$ follows immediately for $P$.

Now suppose that $X_0 \in P$.
If $P \in \Spec R$ or $P \in \Pspec A$  then $X_0 \ast R = X_0 A \subseteq P$, and  $(1)$ follows by induction, considering the image of $P$ in $\Pspec A/\ang{X_0} = \Spec R/ \ang{X_0}$.

For $(2)$ the case $X_0 \not \in P$  follows from \cite[Corollary~4.4]{Jordan}, applied to $P^{\circ} \subseteq R^\circ = A^\circ$, and the case $X_0\in P$ follows by induction, as above.
\end{proof}

We  turn now to describing the topological space $\Spec R =  \Pspec A$.  
We note that this space has a natural stratification:
for $0 \leq j \leq n+1$, let
\[ \Spec_j(R) = \{ P \in \Spec(R) \st X_{i} \in P \textrm{ if and only if } 0\leq i<j \}.\]
It is immediate that $\Spec R $ is the disjoint union of the $\Spec_i(R)$.
By Proposition~\ref{prop:factX1}, we have $\Spec_j R \cong \Spec_{j-1} R(n-1, a+1)$.
Thus to describe the primes of $R$ explicitly, it suffices by induction to describe the open stratum $\Spec_0 R$.
We have $\Spec_0(R) \cong \Spec R^\circ = \Pspec A^\circ$, using (the proof of) Theorem~\ref{thm:spectra}.

Before describing $\Pspec A^\circ = \Spec R^\circ$, we  establish some notation. 
 Let $K$ be a commutative ring.  By $\PP_K(2, \dots, n)$ we denote the weighted projective space $\PP(2, \dots n)$ with base $\Spec K$:  explicitly, 
$\PP_K(2, \dots, n) = \Proj K[Y_2, \dots, Y_n]$ (see \cite[page 76]{Hartshorne}), where $K$ is assumed to be concentrated in degree $0$, and $\deg Y_i = i$.

If $C$ is a graded ring, let $\Spec_{\gr}(C) $ be the set of prime graded ideals of $C$, under the Zariski topology; if $C$ has multiple gradings, say $d$ and $e$, we will write $\Spec_{d-\gr}(C)$ or $\Spec_{e-\gr}(C)$ to indicate which grading is being used.
Likewise, let $\Pspec_{d-\gr}(A^\circ) = \{ \mbox{ $d$-graded Poisson primes of $A^\circ$  } \}$.

For  $n \geq 2$ let $C(n) = \kk[Z_2, \dots, Z_n]$, graded with $\deg Z_i = i$.
Let $X(n) = \Spec_{\gr}C(n)$.
Note that $X(n) = \PP_\kk(2,\dots, n) \sqcup \{ C(n)_+ = \ang{Z_2, \dots, Z_n}\}$ and is $(n-2)$-dimensional.

The structure of $\Pspec A^\circ = \Spec R^\circ$ depends sensitively on the value of $a$, as shown in the next result.

\begin{theorem}\label{thm:spectra2}
Assume that $n \geq 2$.
\begin{enumerate}
\item If $a \not \in \QQ$, then $\Pspec A^\circ$ is homeomorphic to $X(n)$. 
Further, all primes of $R^\circ$ and Poisson primes of $A^\circ$ are $d$-graded.
\item If $a \in \QQ^\times$, then $\Pspec A^\circ$ is homeomorphic to the rational affine variety $\Spec Z$, where $Z$ is the Poisson centre of $A^\circ$, and has dimension $(n-1)$.  
Further, $\Pspec_{d-\gr}(A^\circ)$ is homeomorphic to $X(n)$.
\item If $a=0$, then  $\Pspec A^\circ$ is the disjoint union of a stratum homeomorphic to $\PP_{\CC[X_0^{\pm 1}]}(2,\dots, n)$ and $ \Pspec A^\circ/(Y_2, \dots, Y_n) \cong \AA^2 \ssm V(X_0) $ and has dimension $\max\{n-1,2\}$.
We have that $\Pspec_{d-\gr}(A^\circ)$ is  the disjoint union of a stratum homeomorphic to $\PP(2, \dots, n)$ and a stratum homeomorphic to $\AA^1$.
\item Further, as long as $a \neq 0$, then $\Pspec_{d-\gr}(A^\circ)$ does not depend on $a$:  that is, an ideal $P$ of $\kk[X_0, \dots, X_n]$ is a Poisson prime of some $A(n,a)^\circ$ (with $a \neq 0$) if and only if $P$ is a Poisson prime of all $A(n,a)^\circ$. 
\end{enumerate}
\end{theorem}

We immediately obtain:
\begin{corollary}\label{cor:specsum}
Let $n \geq 1$ and let $a\in \kk$.
Then $\Spec R(n,a) \cong \Pspec A(n,a)$ is a union of quasiprojective rational varieties and has dimension:
\begin{itemize}
\item $\max \{ n-2,1\}$ if $a \not \in \QQ$;
\item $\max \{n-1, 1\}$ if $a \in \QQ$ and $(n,a) \not\in \{ (2,-1), (1,0)\}$;
\item $2$ if $(n,a) \in \{ (2,-1), (1,0)\}$.
\end{itemize}
\end{corollary}
\begin{proof}
Combine Theorem~\ref{thm:spectra2}, Proposition~\ref{prop:factX1}, and Example~\ref{exn=2}.
\end{proof}

\begin{example}
In the case $a \in \QQ^\times$, it is not necessarily true that Poisson primes of $A^\circ$ are centrally generated.  For example, let $n=2$ and $a = -7/4$. Recall that $Y_2 = X_2 - \frac{X_1^2}{2X_0}$.
Then $\Delta(Y_2) = 0$ and $\Gamma(Y_2) = (1/4) Y_2$.  By Proposition~\ref{prop:pnn}, 
 $Y_2$ is Poisson normal in $A^\circ$.

We will see that the Poisson centre of $A^\circ$ is $Z = \CC[X_0 Y_2^7]$, so the Poisson ideal $\ang{Y_2}$ of $A^\circ$ is not centrally generated. However, $Y_2$ is the unique Poisson prime of $A^\circ$ with $\ang{Y_2}\cap Z = \ang{X_0 Y_2^7}$.
\end{example}

The  situation in the example is typical.  Theorem~\ref{thm:spectra2}(2) will follow from:
\begin{proposition}\label{prop:homeo2}
Assume that $a \in \QQ^\times$.
The map 
\begin{align*}
\phi:  \Pspec A^\circ & \to \Spec Z \\
P & \mapsto P \cap Z
\end{align*}
is a homeomorphism.
The inverse map is defined by
\begin{align*}
\psi:  \Spec Z & \to \Pspec A^\circ  \\
Q_0 & \mapsto \sqrt{Q_0 A^\circ}.
\end{align*}
Further,  $\Spec Z$ is a rational variety of dimension $n-1$.
\end{proposition}

As an immediate consequence, we have:
\begin{corollary}\label{cor:one}
If $a\in \QQ^\times$, the map $P \mapsto P[X_0^{-1}] \cap Z$ induces a homeomorphism between $\Pspec(A) \ssm V(X_0) \cong \Pspec(A^\circ)$ and $\Spec(Z)$.
\qed
\end{corollary}

Proposition~\ref{prop:homeo2} will follow from a general lemma on gradings of localised polynomial rings.

\begin{lemma}\label{lem1}
Let $B = \CC[X_0^{\pm1}, X_1, \dots, X_m]$, where $\deg X_i = a_i\in \ZZ$.  Assume that  $a_0 \neq 0$.
Write the $\ZZ$-grading on $B$ as $B=\bigoplus_{n \in \ZZ} B_n$, and let $Z = B_0$.

Let $Q_0 \in \Spec Z $ and let $N = \sqrt{Q_0B}$.  Then:
\begin{enumerate}
\item $N$ is prime;
\item $\Kdim (B/N) = 1+ \Kdim (Z/Q_0) $.
\item The map $Q_0 \mapsto \sqrt{Q_0B}$ gives a homeomorphism $\eta:  \Spec Z\to \Spec_{\gr}(B)$.  The inverse map is given by $Q\mapsto Q \cap Z$.
Further, $\Spec Z$  is a rational variety of dimension $n-1$:  that is, the fraction field $\Frac(Z)$ of $Z$ is isomorphic to $\CC(t_0, \dots, t_{n-2})$.  
\end{enumerate}
\end{lemma}
\begin{proof}
$(1)$.  Let $B' = \bigoplus_{k \in \ZZ} B_{ka_0} \cong Z[X_0^{\pm 1}]$.  
Suppose that $N$ is not prime; then there are $x,y \not\in N$ with $xy \in N$, so $(xy)^m \in Q_0B$ and $(xy)^{mp} \in Q_0B'$ for some $m, p \in \ZZ_{\geq 1}$.   
Thus either $x^{mp}$ or $y^{mp}$ is in $Q_0B$, since $Q_0B' \subseteq Q_0B$ is a prime ideal of $B'$.
This is a contradiction.

$(2)$.  Since $B$ is module-finite over $B'$, therefore $B/Q_0B$ is module-finite over $B'/Q_0B' \cong (Z/Q_0)[X_0^{\pm 1}]$.
Thus $\Kdim B/N = \Kdim B/Q_0B = \Kdim B'/Q_0B' = 1+ \Kdim Z/Q_0$.

$(3)$.  
 Let $\sC = Z \ssm Q_0$ and let $T = B \sC^{-1}/N \sC^{-1}$.  Since $N$ is graded and prime, $T$ is a graded domain.
Let $d = \min \{ k \in \ZZ_{\geq 0} \st T_k \neq 0\}$ and let $0\neq x \in T_d$.  

We claim that $T \cong \Frac(Z/Q_0)[x^{\pm1}]$.
Certainly $T_0 = Z \sC^{-1}/Q_0 \sC^{-1} \cong \Frac(Z/Q_0)$. 
Now,  $X_0$ and $X_0^{-1}$ map to nonzero elements of $T$ under the natural map $B \to T$.  By abuse of notation, let  $X_0$ also denote the image of $X_0$ in $T$.  
A straightforward combinatorial argument  shows that $T = \bigoplus_{k \in \ZZ} T_{kd}$.  In particular, $a_0 = d\ell$ for some $\ell$ and so $x^\ell (X_0)^{-1} \in T_0$ is invertible. 
Thus $x$ is invertible in $T$.  It follows that $T_{dk} = T_0 x^k$ for all $k$, completing the proof of the claim.

Since $T \cong \Frac(Z/Q_0)[x^{\pm1}]$ is a Laurent polynomial ring over a field, it has no non trivial graded ideals.
Now, $\Spec_{\gr} T$ is in bijection with 
\[\{ Q \in \Spec_{\gr} B \st Q \supseteq N, Q \cap \sC = \emptyset\} =
\{ Q \in \Spec_{\gr} B \st Q \cap Z = Q_0 \},\]
where we have used part $(1)$ of the lemma.
Thus there is only one such $Q$, namely $N$.

Define $\theta:  \Spec_{\gr} B \to \Spec Z$ by $\theta(Q) = Q \cap Z$.  The argument above shows that if $Q \in \Spec_{\gr} B$, then $Q = \eta \theta (Q)$.  
Since $\theta\eta$ is easily seen to be the identity on $\Spec Z$, therefore $\theta = \eta^{-1}$.  
As $\eta$ and $\theta$ clearly preserve inclusions, they are continuous and thus homeomorphisms.

Since $Z$ is a normal semigroup algebra, $\Spec Z$ is  rational.
\end{proof}

We next give an explicit characterisation of Poisson primitive ideals of $A^\circ$ if $a\in \QQ$.  
By Theorem~\ref{thm:spectra} these are the same as the primitive ideals of $R^\circ$.
Recall from the end of Section \ref{FIRSTPROPS} that we define a grading $A^\circ=\bigoplus_{e\in\kk}A^\circ(e)$ by setting $A^\circ(e)$ to be the $e$-eigenspace of $\Gamma_a$.
Moreover notice that $(A^\circ)^{\Delta}$ is an $e$-graded subalgebra of $A^\circ$. 

\begin{corollary}\label{cor:glasgow}
If $a \in \QQ^\times$ then $\Pspec A^\circ$ is homeomorphic to $\Spec Z$.  Furthermore, a Poisson prime $P$ of $\Pspec A^\circ$ is Poisson primitive if and only if $P \cap Z$ is a maximal ideal of $Z$.
\end{corollary}

\begin{proof}
Let $B = \kk[X_0^{\pm1}, Y_2, \dots, Y_n] = (A^\circ)^{\Delta}$.
Let $Z$ be the Poisson centre of $A^\circ$.
By Proposition~\ref{prop:pnn}, $Z = B^{\Gamma_a}=B(0)$. 
Since $a \neq 0$, $X_0 \in \delta(B)$ so no nontrivial ideal of $A^\circ$ can contain $\delta(B)$.  
By Proposition~\ref{prop:jordan},  there is an inclusion-preserving bijection $\rho$ between $\Pspec A^\circ$ and the set of $\delta$-prime ideals of $B$, defined by $\rho(P) = P\cap B$.  
Note that a $\delta$-prime of $B$ is the same as an $e$-graded prime of $B$.  Thus by Lemma~\ref{lem1} the map $\phi = \theta\rho:  \Spec A^\circ \to \Spec Z$ is a homeomorphism. 
The inverse to $\phi$ is 
\[ \psi = \rho^{-1} \eta:  Q_0 \mapsto A^\circ(\sqrt{Q_0B}) = \sqrt{Q_0A^\circ}.\]
By  \cite[Corollary~4.4]{Jordan}, $P \in \Pspec A^\circ$ is Poisson primitive if and only if $P \cap B$ is a {\em $\delta$-primitive} ideal of $B$, i.e. $P \cap B$ is the largest $\delta$-stable ideal of $B$ contained in some maximal ideal $M$.
Since an ideal of $B$ is $\delta$-stable $\iff$ it is $\Gamma_a$-stable,  a $\delta$-primitive ideal of $B$ is a maximal  $e$-graded ideal of $B$, and by Lemma~\ref{lem1} these are precisely ideals of the form $\psi(M_0)$ for $M_0 = P \cap B \cap Z = P\cap Z \in \operatorname{maxspec} Z$.
\end{proof}

We next assume that $a \neq 0$, and consider $d$-graded Poisson primes of $A^\circ$, which by Theorem~\ref{thm:spectra} may be identified with $d$-graded primes of $R^\circ$.
For $2 \leq i \leq n$, let $Z_i = Y_iX_0^{-1}$ and $C= \kk[Z_2, \dots, Z_n]$.
The $e$-grading on $A^\circ$ restricts to $C$, with $e(Z_i) = i$.

\begin{proposition}\label{prop:specgr}
If $a \neq 0$, then $\Pspec_{d-\gr}(A^\circ) $ is homeomorphic to $\Spec_{e-\gr}(C) \cong X(n)$.
Further, as long as $a \neq 0$, then $\Pspec_{d-\gr}(A^\circ)$ does not depend on $a$ in the sense of Theorem~\ref{thm:spectra2}(4). 
\end{proposition}
\begin{proof}
Note that $C$ is precisely $B_0 = \{ b\in B \ | \ \mbox{ $b$ is $d$-homogeneous of degree $0$ } \}$. 
As in the proof of Corollary~\ref{cor:glasgow}, there is an inclusion-preserving bijection between $\Pspec(A^\circ)$ and $\Spec_{e-\gr}(B)$, and it is clear that this bijection takes $d$-graded primes to $d$-graded primes:  in other words, $\Pspec_{d-\gr}(A^\circ)$ is homeomorphic to the set of $(d,e)$-bigraded primes of $B$, via the map $P \mapsto P \cap B$.
Since   $B=C[X_0^{\pm1}]$ is strongly $d$-graded, we have $\Spec_{d-\gr}(B) \cong \Spec C$.
Thus $\Pspec_{d-\gr}(A^\circ)$ is homeomorphic to $\Spec_{e-\gr}(C)$ via $P \mapsto P \cap C$.  
But $\Spec_{e-\gr}(C)$ is homeomorphic by definition to $X(n)$.

For the final statement, note that the definition of $C$, the $e$-grading on $C$, and the restriction homeomorphism $\Pspec_{d-\gr}(A^\circ) \to \Spec_{e-\gr} (C ) $ do not depend on the value of $a$.
\end{proof}

The result above is particularly strong if $a \not\in \QQ$, since we then have:
\begin{proposition}\label{prop:allgraded}
If $a \not \in \QQ$, then all Poisson primes of $A^\circ$ are $d$-graded.
\end{proposition}

\begin{proof}
We have seen that $\Pspec(A^\circ) \cong \Spec_{e-\gr}(B)$, where $B = \kk[X_0^{\pm1}, Y_2, \dots, Y_n]$.
Let $\mb S = \{ e(b) \ | \ b \in B\}$, which is clearly equal to $\ZZ a + \NN$.
Since $a \not\in \QQ$, we have $\mb S \cong \ZZ \oplus \NN$ as a semigroup, and it follows that if $b\in B$ is $e$-homogeneous then  $b$ is  $d$-homogeneous.
\end{proof}

We now combine the previous results to prove Theorem~\ref{thm:spectra2}.
\begin{proof}[Proof of Theorem~\ref{thm:spectra2}]
$(1)$ follows from Propositions~\ref{prop:specgr} and \ref{prop:allgraded}, and $(2)$ is Corollary~\ref{cor:glasgow} and Proposition~\ref{prop:specgr} again.

For $(3)$, let $B_+ = \sum B Y_i$; we have $B_+ = \delta(B)B$.  Note that by Proposition~\ref{prop:jordan}, $\Pspec A^\circ$ is equal to the disjoint union of:
\beq\label{aa}
\{ P = QA^\circ \ | \ \mbox{ $Q \not\supseteq B_+ $ is a $\delta$-prime ideal of $B$} \}
\eeq
and
\beq \label{bb}
\{ P \ | \ P \mbox{ is a prime ideal of $A^\circ$ with } P \supseteq B_+\}.
\eeq
As we have repeatedly seen above, $\delta$-prime ideals of $B$ are the same as $e$-graded ideals:  that is, 
\eqref{aa} is homeomorphic to
\[ \{ Q \in \Spec_{e-\gr} B \ | \ Q \not\supseteq B_+\}\cong \Proj B = \PP_{\CC[X_0^{\pm1}]}(2, \dots, n).\]
On the other hand, \eqref{bb} is clearly homeomorphic to the  spectrum of $A^\circ/(Y_2, \dots, Y_n) \cong \kk[X_0^{\pm 1}, X_1]$ (with trivial Poisson bracket).

Finally,  $(4)$ follows from Proposition~\ref{prop:specgr}.
\end{proof}

\begin{corollary}\label{cor:specgr}
If $a \not \in \{ -n +1, \dots, -1, 0\}$ then $\Pspec_{d-\gr}A(n,a) = \Spec_{d-\gr} R(n,a)$ does not depend on $a$. In particular $\Spec R(n,a)\cong\Spec R(n,b)$ for any $a,b\notin\QQ$. 
\end{corollary}
\begin{proof}
Combine Theorem~\ref{thm:spectra2}(4), Proposition~\ref{prop:factX1}, and induction. The second assertion follows from Proposition~\ref{prop:allgraded}.
\end{proof}


\section{Dixmier-Moeglin equivalences and skewfields}
\label{DME}

	In this section we show that the algebra $R(n,a)$ satisfies the Dixmier-Moeglin equivalence (DME). Recall that primitive ideals are annihilators of simple modules. In particular they are not easily distinguished among the primes. The DME  characterises them with algebraic and topological properties. A prime ideal $P$ in a noetherian ring $R$ is said {\em rational} provided that the field $Z(\F R/P)$ is algebraic over the ground field, and is said {\em locally closed} if the point $\{P\}$ is locally closed in $\Spec R$ (with respect to the Zariski topology). We say that the {\em Dixmier-Moeglin equivalence} holds for a given noetherian algebra if the sets of primitive ideals, locally closed ideals and rational ideals coincide. This idea originated in the work of Dixmier and Moeglin who showed that for any finite dimensional complex Lie algebras, theses sets are equal. 

Thanks to our result on the DME we prove a transfer result which says that the Poisson algebra $A(n,a)$ satisfies a similar equivalence, the so-called Poisson Dixmier-Moeglin equivalence (PDME), which we will recall later. To prove our results on the DME we relate $R(n,a)$ to the enveloping algebra of a solvable Lie algebra sitting inside the localisation $R(n,a)^\circ$. With the Gelfand-Kirillov conjecture \cite[Section 5]{GK} in mind, this motivated us to investigate  the skewfield of fractions of $R(n,a)$ at the end of this section.

Let $R=R(n,a)$ and recall that $R^{\circ}=R[X_0^{-1}]$. We denote by $T$ the subalgebra of $R^{\circ}$ generated by $X_0,X_1,Y_2,\dots,Y_n$. Setting $C:=\kk[X_0,Y_2,\dots,Y_n]$ we have that $T=C[X_1;\de]$, where $\de=X_0\Gamma_a$ is a derivation of the commutative ring $C$. Moreover we set $Y_0:=X_0$ and we denote by $\mathfrak{g}_a$ the $n$-dimensional solvable Lie algebra with basis elements $Y_0,X,Y_2,\dots,Y_n$ and Lie brackets 
\begin{equation}
\label{lie}
[X,Y_i]=(a+i)Y_i \qquad [Y_i,Y_j]=0 \qquad\text{for all }i,j.
\end{equation} 
In particular sending $Y_i$ to $Y_i$ and $X$ to $Y_0^{-1}X_1$ we obtain the isomorphism 
\begin{equation}
\label{uiso}
U(\mathfrak{g}_a)[Y_0^{-1}]\cong T[Y_0^{-1}]\cong R^{\circ}.
\end{equation}

We now state the main result of this section.

\begin{theorem}
\label{thm:DME}
The algebra $R(n,a)$ satisfies the DME for any $a\in\kk$ and $n\geq1$.
\end{theorem}

The proof of Theorem \ref{thm:DME} relies on the following lemmas.

\begin{lemma} 
\label{quo}
Let $P\in\Spec R$ and suppose that $X_0\in P$. Then $P$ is locally closed in $\Spec R$ if and only if $P/\langle X_0 \rangle$ is locally closed in $\Spec R/\langle X_0 \rangle$.
\end{lemma}

\begin{proof}
This follows from the isomorphism $\frac{R/\langle X_0 \rangle}{P/\langle X_0 \rangle}\cong\frac{R}{P}$.
\end{proof}

For $P\in\Spec R$ with $X_0\notin P$, we set $P^{\circ}:=P[X_0^{-1}]\in\Spec R^{\circ}$.

\begin{lemma}
\label{loc}
Let $P\in\Spec R$ and suppose that $X_0\notin P$. Then $P$ is locally closed in $\Spec R$ if and only if $P^{\circ}$ is locally closed in $\Spec R^{\circ}$.
\end{lemma}

\begin{proof}
It is enough to prove the lemma for $P=\ang{0}$. 
Recall that $\ang{0}$ is locally closed if and only if $\ang{0} \neq \bigcap \{ Q \in \Spec R \st  Q \neq \ang{0}\}$. 
We set
\[\mathcal{I}_1=\bigcap_{\ang{0}\neq Q\in\Spec R}Q \quad \text{and} \quad \mathcal{I}_2=\bigcap_{\ang{0}\neq T\in\Spec R^{\circ}}T.\]
Suppose that $\mathcal{I}_1\neq \ang{0}$ and let $\ang{0}\neq U\in \mathcal{I}_1$. Then $U\in\mathcal{I}_2$ and $\mathcal{I}_2\neq 0$. Reciprocally if $\mathcal{I}_2\neq \ang{0}$, then there exists $U$ nonzero inside $\bigcap_{\ang{0}\neq Q\in\Spec R\ \text{and }X_0\notin Q}Q$ (recall that prime ideals are completely prime in $R$, see assertion (1) of Theorem \ref{thm:spectra}). Then $UX_0$ belongs to any nonzero $Q\in\Spec R$ and $\mathcal{I}_1\neq \ang{0}$.
\end{proof}

We can now prove Theorem \ref{thm:DME}.

\begin{proof}[Proof of Theorem \ref{thm:DME}]
We proceed by induction on $n$. It is well-known that the Jordan plane $R(1,a)$ ($a\neq0$) and the commutative polynomial ring $R(1,0)$ satisfy the DME, so that the base case $n=1$ is true. Suppose that $R(n-1,b)$ satisfies the DME for any $b\in\kk$ and $n>1$. 
By \cite[II.7.17]{BG} the algebra $R$ satisfies the (noncommutative) Nullstellensatz. Then by \cite[II.7.15]{BG} we have the implications locally closed $\implies$ primitive $\implies$ rational. It remains to prove that rational implies locally closed.

	Let $P\in\Spec R$ be rational. Suppose first that $X_0\in P$. Then
\[Z\left( \F \frac{R/\langle X_0 \rangle}{P/\langle X_0 \rangle}\right)\cong Z\left( \F \frac{R}{P}\right)\]
and $P/\langle X_0 \rangle$ is rational in $R/\langle X_0 \rangle$. Since $R/\langle X_0 \rangle\cong R(n-1,a+1)$ satisfies the DME by the induction hypothesis, the prime $P/\langle X_0 \rangle$ is locally closed. We conclude that $P$ is locally closed by Lemma \ref{quo}.

	Suppose now that $X_0\notin P$. Then 
	\[Z\left(\F \frac{R^{\circ}}{P^{\circ}}\right)\cong Z\left( \F \frac{R}{P}\right)\]
and $P^{\circ}$ is rational in $R^{\circ}$. Since the algebra $U(\mathfrak{g}_a)$ satisfies the DME over $\kk$ by \cite{IS}, the localisation $U(\mathfrak{g}_a)[X_0^{-1}]\cong R^{\circ}$ satisfies the DME. Then $P^{\circ}$ is locally closed in $\Spec R^{\circ}$ and we conclude that $P$ is locally closed in $\Spec R$ by Lemma \ref{loc}.
\end{proof}

The second main theorem of the section is that $A(n,a)$ satisfies the Poisson Dixmier-Moeglin equivalence, which we define here. Recall from Section~\ref{PRIMES} that a Poisson primitive ideal is by definition the largest Poisson ideal contained inside a maximal ideal. Let $A$ be a Poisson $\kk$-algebra and $P\in\Pspec(A)$. The ideal $P$ is said \emph{locally closed} if the point $\{P\}$ is a locally closed point of $\Pspec(A)$ and is said \emph{Poisson rational} provided the field $Z_P\big(\F(A/P)\big)$ is algebraic over the ground field $\kk$. We say that the \emph{Poisson Dixmier-Moeglin equivalence} holds for the Poisson algebra $A$ if the sets of Poisson primitive ideals, of locally closed Poisson ideals and of Poisson rational ideals coincide. Our proof proceeds via a transfer result for the PDME.  

\begin{theorem}
\label{pdme}
The algebra $A(n,a)$ satisfies the PDME for any $a\in\kk$ and $n\geq1$.
\end{theorem} 

 Recall that we set $A^{\circ}:=A[X_0^{-1}]$. We first prove two lemmas.  
 An algebra is {\em catenary} if for every pair of distinct prime ideals $P\subset Q$ all saturated chains of prime ideals from $P$ to $Q$ have the same length.

\begin{lemma}
\label{ratt}
Let $S$ be a catenary and noetherian $\kk$-algebra with finite GK dimension and let $P\in\Spec S$. Then 
\[\left\{ Q\in\Spec S \ |\ ht(Q)=ht(P)+1 \right\}\text{ is finite}\quad\implies\quad P\text{ is locally closed,}\]
where $ht(P)$ is the height of $P$, i.e. the supremum of the length of chains of prime ideals descending from $P$. 
\end{lemma}

\begin{proof}
Since $S$ is catenary we can assume that $P=\ang{0}$. Suppose that $S$ has only finitely many height one prime ideals, namely, $P_1,\dots,P_{\ell}$. Since $S$ has finite GK dimension, it satisfies the DCC on prime ideals. In particular any nonzero $P\in\Spec S$ contains one of the $P_i$  and therefore $\{\langle0\rangle\}=\Spec S\setminus V(\cap_i P_i)$ is open in its closure.
\end{proof}

\begin{lemma}
\label{catenary}
The ring $R=R(n,a)$ is catenary for any $a\in\kk$ and $n\in\ZZ_{>0}$.
\end{lemma}

\begin{proof}
Thanks to \cite[II.9.5]{BG}, Lemma \ref{cor:noeth} and the proof of Theorem \ref{thm:ASreg}, we only need to prove that $\Spec R$ has normal separation, that is, for every distinct pair $P\subset Q$ of comparable primes in $R$ the ideal $Q/P$ of $R/P$ contains a nonzero normal element of $R/P$.

 We proceed by induction on $n$. Let $P\subset Q$ be a pair of comparable primes in $R$. The proof split into three cases. 

First assume that $X_0\in P$, so that $X_0\in Q$. Then $P/\langle X_0 \rangle \subset Q/\langle X_0 \rangle $ inside $R/\langle X_0 \rangle\cong R(n-1,a+1)$, and we are done by induction.

Next assume that $X_0\notin Q$, so that $X_0\notin P$. Then $P^{\circ}\subset Q^{\circ}$ inside $R^{\circ}\cong U(\mathfrak{g}_a)[X_0^{-1}]$. Since $\mathfrak{g}_a$ is solvable $\Spec U(\mathfrak{g}_a)$ has normal separation by \cite[Theorem 12.19]{GW}. It is then easy to see that $\Spec U(\mathfrak{g}_a)[X_0^{-1}]\cong R^{\circ}$ has normal separation. Hence there exists a nonzero normal element $U$ in $Q^{\circ}/P^{\circ}\cong Q/P[X_0^{-1}]$ (where we denote again by $X_0$ its image in the quotient $Q/P$). In particular there is an integer $\ell\geq0$ such that 
$UX_0^{\ell}$ is a nonzero normal element in $Q/P$ (recall that $X_0$ is normal in $R$).

Finally suppose that $X_0 \notin P$ and $X_0\in Q$. Then $X_0\in Q\setminus P$ and is normal modulo $P$ as it is already normal in $R$. 
\end{proof}

We now prove Theorem \ref{pdme}.

\begin{proof}[Proof of Theorem \ref{pdme}]
By \cite[Propositions 1.7, 1.10]{Oh} we have the implications Poisson locally closed $\implies$ Poisson primitive $\implies$ Poisson rational. Let $P$ be a Poisson rational ideal of $A$. Then by \cite[Theorem 8.3]{BLLM} the set $\left\{ Q\in\Pspec A \ |\ ht(Q)=ht(P)+1 \right\}$ is finite. Since $\Pspec A= \Spec R$, the set
\[\left\{ Q\in\Spec R \ |\ ht(Q)=ht(P)+1 \right\}=\left\{ Q\in\Pspec A \ |\ ht(Q)=ht(P)+1 \right\}\]
is finite. By Lemma \ref{ratt} we conclude that $P$ is locally closed in $\Spec R$, hence $P$ is locally closed in $\Pspec A$ since $\Pspec A= \Spec R$.
\end{proof}





To end the section we investigate the structure of the skewfield of fractions $\F R$ of the noetherian domain $R$. 

Recall the definition of the solvable Lie algebra $\mf g_a$ from \eqref{lie}. By equation (\ref{uiso}), $R^{\circ}$ is isomorphic to a localisation of the enveloping algebra  $U(\mathfrak{g}_a)$. Then by \cite{Bo,Jo,MC}, the skewfield $\F R$ is isomorphic to a Weyl skewfield, when $\kk$ is algebraically closed and when the Lie algebra $\mathfrak{g}_a$ is algebraic. From \cite[Section 8]{GK} we note that this algebra is algebraic if and only if $a\in\Q$. In that case we provide an explicit description of this Weyl skewfield, and we show that $\F R$ is not isomorphic to a Weyl skewfield when $a\notin \Q$. Moreover we prove these results over a field of characteristic 
zero that is not necessarily algebraically closed. Recall that $R^{\circ}=B[X_1;X_0\Gamma_a]$, where $B=\C[X_0^{\pm1},Y_2,\dots,Y_n]$. 
  
For any field $K$ and $\varepsilon, a \in K$, let $\mathfrak{h}_{\varepsilon, a}(K)$ be the $3$-dimensional solvable Lie algebra over $K$ with basis $\{x,y,z\}$ and Lie bracket $[x,y]=\varepsilon y$, $[x,z]=az$ and $[y,z]=0$.

\begin{proposition}\label{prop:7.1}
We have $\F R\cong \F U(\mathfrak{g}_{a})\cong\F U(\mathfrak{h}_{\epsilon,a}(K))$
where $K=\big(Q_{gr}(B)\big)_0$ is a field of transcendence degree $n-2$ over $\C$, and where $\varepsilon=2$ if $n=2$ and $\varepsilon=1$ if $n>2$.
\end{proposition}

\begin{proof}
We define a $\Z^2$-grading $f$ on $B$ as follows. Set $f(X_1^{\pm1})=(\pm1,0)$ and $f(Y_i)=(1,i)$ for all $i=2,\dots,n$. This is a combination of the $d$-grading and the $\epsilon$-grading. Note that if $u\in B$ has degree $(0,0)$, then $u\in\ker\Gamma_a$. We now form the graded quotient ring $E:=Q_{gr}(B)$ of $B$ by inverting all its homogeneous elements
\[E=B[h^{-1}\ |\ h\ f\text{-homogeneous}].\]
Note that $E$ is also $\Z^2$-graded by a grading denoted $f$ again. It is a standard fact that $E\cong K[s^{\pm1},t^{\pm1}]$, where $K=E_{(0,0)}$ is a field such that
\[\td_\C K=\td_\C B -2=n-2,\]
and where $0\neq s\in E_{(1,0)}$ and $0\neq t\in E_{(0,\varepsilon)}$. By definition $\varepsilon:=\min\{\alpha\geq 1\ |\ E_{(0,\alpha)}\neq 0\}$.
 For instance we can choose $s=X_0$ and $t=Y_2X_0^{-1}$ when $n=2$, or $t=Y_3Y_2^{-1}$ when $n>2$. Thus we have
\[\F R=\F E[X_1;X_0\Gamma_a]=\F K[s^{\pm1},t^{\pm1}][X_1;X_0\Gamma_a].\]
For $u\in K$ we have $X_1u-uX_1=\Gamma_a(u)=0$ since $f(u)=(0,0)$, and $u$ commutes also with $s$ and $t$ since $E$ is commutative. 
Moreover we get
\begin{align*}
& X_1s-sX_1=X_0\Gamma_a(X_0)=aX_0^2 \\
& X_1t-tX_1=X_0\Gamma_a(t)=\varepsilon X_0t.
\end{align*}
We conclude by setting $X_1':=X_1X_0^{-1}$ that $X_1's-sX_1'=aX_0=as$ and $ X_1't-tX_1'=\varepsilon t$. Since $s,t\in B$ we have $st=ts$ and the result follows by setting $x:=X_1'$, $y:=t$ and $z:=s$.
\end{proof}

\begin{rem}
\label{n-3}
Generators of the field $K$ can be obtained by solving the system of equations
\[\begin{cases} u_0+u_2+\cdots+u_n=0, \\ 2u_2+3u_3+\cdots+nu_n=0. \end{cases}\]
Setting $Y_0:=X_0$ and $Z_i:=Y_iY_{n-1}^{-(n-i)}Y_{n}^{n-i-1}$ for $i\in\{0,2,\dots,n-2\}$ we have $K=\C\left(Z_0,Z_2,\dots,Z_{n-2}\right)$.
Note that $K$ is purely transcendental over $\kk$ of transcendence degree $n-2$. 
\end{rem}

\begin{theorem}
\label{thm:skewfield}
When $a\in\Q$ the skewfield $\F R$ is isomorphic to the first Weyl skewfield $D_1(F)$ over a field $F$ of transcendence degree $n-1$ over $\kk$. When $a\notin \Q$, we have $Z(\F R)=K$, and $\F R$ is not isomorphic to any Weyl skewfield.
\end{theorem}

\begin{proof}
The  skewfield of the enveloping algebra of the Lie algebra $\mathfrak{h}_{\varepsilon,a}(\kk)$ is either isomorphic to the first Weyl skewfield over a field of transcendence degree 1 over $\kk$ when $a$ is rational, or has a trivial centre when $a$ is irrational. This is a classical fact that can be found in \cite[Section 8]{GK} and \cite[Proposition 1.2.4.1 and Remark after]{Richard}. The result follows from Proposition~\ref{prop:7.1} and the isomorphism $U(\mathfrak{h}_{\varepsilon, a}(K))\cong U(\mathfrak{h}_{\varepsilon, a}(\kk))\otimes_\kk K$.
\end{proof}

\begin {rem}
\label{ff}

$(1)$ Let $a\in\Q$ and suppose that $a=p/q\in\Q$ with $\gcd(p,q)=1$ (when $a=0$ we have $p=0$ and we set $q=1$). Setting $Z:=X_0^{q}Y_2^{p}Y_3^{-p}$ we obtain
\[F=\C\left(Z,Z_0,Z_1,\dots,Z_{n-2}\right)\]
where the $Z_i$'s are defined in Remark \ref{n-3}. 

$(2)$ For $a,b\in\Q$ it is clear that $\F R(n,a)\cong\F R(n,b)$ since both are isomorphic to a Weyl skewfield over a field of transcendence degree $n-2$. When $a\notin \Q$ and $b=\pm a+n$ for some $n\in\Z$, it is easy to verify that the skewfields $\F R(n,a)$ and $\F R(n,b)$ are isomorphic. However it remains unclear whether or not this condition is also necessary for an isomorphism $\F R(n,a)\cong\F R(n,b)$ when $a,b\notin\Q$.
\end{rem}

Using similar methods we can prove the following result about the Poisson structure of the field $\F(A)$. For a Lie algebra $\mathfrak{g}$ we denote by $S(\mathfrak{g})$ its symmetric algebra that we endow with the so-called Kirillov-Kostant-Souriau Poisson bracket, that is the Poisson bracket obtained by extending by bi-derivation and bi-linearity the Lie bracket of $\mathfrak{g}$ inside $S(\mathfrak{g})$
\[\{X,Y\}:=[X,Y]_{\mathfrak{g}}\]
for any $X,Y\in\mathfrak{g}$.

\begin{theorem}\label{thm:7.5}
As Poisson algebras, we have $\F A\cong\F S(\mathfrak{g}_{a})\cong \F S(\mathfrak{h}_{\varepsilon,a}(K))$. Moreover, if $a\in\Q$ then $\F A$ is isomorphic to the field of fractions of the Poisson algebra $\mathcal{P}=F[X,Y]$, where $\{X,Y\}=1$ and where $F$ is the field described in Remark \ref{ff}. When $a\notin\Q$, the Poisson centre of $\F A$ is $K$ and $\F A$ is not isomorphic to $\F \mathcal{P}$.
\end{theorem}


\section{Examples of spectra}\label{EXAMPLES}

 In this final section we study $\Pspec A(n,a)$ and $\Spec R(n,a)$ for small values of $n$. Since these two spectra are equal it is enough to describe $\Pspec A(n,a)$.
  Because we will give explicit description of these spectra we assume that $\kk$ is algebraically closed in this section.
 \begin{example}\label{n=2}
Let $n=1$ and consider $A= A(1,a)$.  
If $a=0$ then $A(1,a)=\C[X_0,X_1]$ with trivial Poisson bracket and thus
\[\Pspec A=\Spec \C[X_0,X_1].\] 
Suppose that $a\neq0$. Then $A$ is isomorphic to the Poisson-Jordan plane (see Example \ref{exn=2}). It is well-known that its Poisson spectrum is
\[\{\langle 0 \rangle,\langle X_0\rangle,\langle X_0,X_1-\la\rangle\ |\ \la\in\C\}\]
and that the $d$-graded primes are $\langle X_0\rangle,\langle X_0,X_1\rangle$. Moreover only $\langle X_0\rangle$ is not Poisson primitive.
\end{example}

 Assume that $n\geq2$ and that $a \in \QQ^\times$. We denote by $\Pspec_1 A(n,a)$ the set of Poisson prime ideals of $A$ that contain $X_0$ and by $\Pspec_0 A(n,a)$ the set of Poisson prime ideals of $A$ that do not contain $X_0$. 
 Since $A(n,a)/\ang{X_0} \cong A(n-1,a+1)$ by Proposition \ref{prop:factX1} there is a homeomorphism:
\[\Pspec_1 A(n,a) \cong\Pspec A(n-1,a+1).\]
On the other hand by Theorem \ref{thm:spectra2} there is a homeomorphism
\[\Pspec_0 A(n,a)\cong\Spec(Z),\]
where $Z$ is the Poisson centre of $A^\circ = A(n,a)[X_0^{-1}]$. Thus to describe $\Pspec A(n,a)$ completely we must study the ring $Z$.
We know that $\Kdim Z = n-1$, so we first construct $n-1$ algebraically independent elements of $Z$.

Write $a=p/q$ with $\gcd(p,q)=1$ and $-p>0$. 
For $i=2,\dots,n$ let $d_i=\gcd(p,i)>0$, and  set $u_i=\frac{p+iq}{d_i}$ and $v_i=\frac{-p}{d_i}$. Note that $\gcd(u_i,v_i)=1$.
Finally we set
\begin{equation}
\label{yi'}
Y_i'=X_0^{u_i}Y_i^{v_i}.
\end{equation}
By construction $e(Y_i')=0$, and we have $Z':=\C[Y_2',\dots,Y_n']\subseteq Z$.

 We can be more precise
\begin{lemma}\label{lem:freeZ}
$Z$ is a free $Z'$-module with basis
\begin{equation}
\label{s}
S=\left\{X_0^{s}Y_2^{s_2}\cdots Y_n^{s_n}\ |\ 0\leq s_i<v_i,\ i=2,\dots,n\ \text{and}\ ps+\sum_{i=2}^{n}u_id_is_i=0\right\}.
\end{equation}
In particular its rank is $|S|\leq \prod_{i=2}^{n}v_i$.
\end{lemma}

\begin{proof}
By the extension of Proposition~\ref{prop:pnn} to $A^\circ$ we have $Z=(A^\circ)^{\Gamma_a, \Delta} = B_0$ and so
\begin{align*}
Z&=\text{Span}\left\{X_0^{s}Y_2^{s_2}\cdots Y_n^{s_n}\ |\ s\in\Z,\ s_2,\dots,s_n\in\N,\ \text{and}\ as+\sum_{i=2}^{n}(a+i)s_i=0\right\}\\
&=\text{Span}\left\{X_0^{s}Y_2^{s_2}\cdots Y_n^{s_n}\ |\ s\in\Z,\ s_2,\dots,s_n\in\N,\ \text{and}\ ps+\sum_{i=2}^{n}u_id_is_i=0\right\}.
\end{align*}
Let $M=X_0^{s}Y_2^{s_2}\cdots Y_n^{s_n}\in Z$ and for $i=2,\dots,n$ set $s_i=v_i s_i'+\varepsilon_i$, where $s_i'\geq0$ and $0\leq\varepsilon_i<v_i$. Then we rewrite
\[M=\big(Y_2'^{s_2'}\cdots Y_n'^{s_n'}\big)\big(X_0^{s-\Omega}Y_2^{\varepsilon_2}\cdots Y_n^{\varepsilon_n}\big),\]
where $\Omega=\sum_{i=2}^n s_i'u_i$. 
Since $M, Y_2'^{s_2'}\cdots Y_n'^{s_n'}\in Z'$ we have $X_0^{s-\Omega}Y_2^{\varepsilon_2}\cdots Y_n^{\varepsilon_n}\in Z$. In particular this shows that a generating set of $Z$ as a module over $Z'$ is given by the set $S$. The result follows since the elements of $S$ are linearly independent over $Z'$.
\end{proof}

Note that it is possible that $S=\{1\}$. For example, let $n=2$ and set $M=X_0^{s}Y_2^{s_2}\in S$. Then $ps+u_2d_2s_2=0$, i.e. $v_2s=u_2s_2$ and  $v_2$ must divide $s_2$ since $\gcd(u_2,v_2)=1$. This implies that $s_2=0$ since $0\leq s_2<v_2$. Therefore $Z=Z'=\C[Y_2']$ when $n=2$.

We next work out $\Pspec A(2,a)$ explicitly.
 By Example~\ref{n=2}, for $a,b\in\Q\ssm\{-1\}$ the sets $\Pspec_1 A(2,a)$ and $\Pspec_1 A(2,b)$ are homeomorphic. More precisely we have
\[\Pspec_1 A(2,a)\cong\Pspec_1 A(2,b)\cong\{\langle X_0\rangle,\langle X_0,X_1\rangle,\langle X_0,X_1,X_2-\mu\rangle\ |\ \mu\in\C\}.\]
The following result explicitly describes the stratum $\Pspec_0 A(2,a)$.

\begin{proposition}\label{prop:Pspec_0}
For $a\in\Q^{\times}$ we have $\Pspec_0 A(2,a)=\{\langle 0\rangle,\langle X_0Y_2\rangle,P_\la\ |\ \la\in\C^{\times}\}$, where
\begin{align*}
P_\la=\left\{
\begin{array}{ll}
\langle X_0^{u_2}Y_2^{v_2}-\la\rangle \qquad\qquad\quad    &-1\leq a<0,\\
\langle (X_0Y_2)^{v_2}-\la X_0^{v_2-u_2}\rangle\qquad\qquad  &a<-1\text{ or }a>0.
\end{array}
\right.
\end{align*}
The ideal $P_\lambda$ is Poisson maximal if and only if $-1 < a < 0$.
\end{proposition}

\begin{proof}
Let $A = A(2,a)$  and let $A^\circ = A[X_0^{-1}]$.  
Let $Z$ be the Poisson centre of $A^\circ$.
Recall from the discussion after Lemma~\ref{lem:freeZ} that $Z=\C[Y_2']$ where $Y_2'=X_0^{u_2}Y_2^{v_2}$. We denote by $\varphi$ the bijection between $\Spec Z$ and $\Pspec_0 A$ sending $Q\in\Spec Z$ to $\varphi(Q)=\sqrt{QA^\circ}\cap A$. It is clear that $\varphi(Y_2'Z)= \langle X_0Y_2\rangle$. If $\la\neq0$, the ideal $(Y_2'-\la)A^\circ$ of $A^\circ$ is prime and so $\varphi((Y_2'-\la)Z)=\big((Y_2'-\la)A^\circ\big)\cap A$. Since $Y_2'=X_0^{u_2-v_2}(X_0Y_2)^{v_2}$, we have
\begin{align}\label{chris}
P_\la:=\varphi((Y_2'-\la)Z)=\left\{
\begin{array}{ll}
\langle Y_2'-\la\rangle \qquad\qquad\quad    &u_2-v_2\geq0\\
\langle (X_0Y_2)^{v_2}-\la X_0^{v_2-u_2}\rangle\qquad\qquad  &\text{else}.
\end{array}
\right.
\end{align}
But $u_2-v_2=\frac{2}{d_2}(p+q)$ and we have:
\[(u_2-v_2)\geq0\iff -1\leq a<0.\]
Thus $\Pspec A$ is as described.

If $-1<a<0$ then $u_2-v_2 >0 $ and $P_\lambda = \ang{X_0^{u_2-v_2}(X_0Y_2)^{v_2} - \lambda}$ is comaximal with $X_0$ and thus with all primes in $\Pspec_1(A)$.  If $a=-1$ then $P_\lambda = \ang{X_0X_2-X_1^2/2 - \lambda}$ is contained in the Poisson ideal $\ang{X_0, X_1^2/2 - \lambda}$.  If $a<-1$ or $a>0$ then $v_2-u_2 > 0 $ and $P_\lambda$ is clearly contained in $\ang{X_0, X_1}$.
\end{proof}

We deduce the following result.

\begin{theorem}\label{thm:spectra3}
Let $a\in\Q^{\times}\ssm\{-1\}$. Then we have
\begin{align*}
\Pspec A(2,a)=\{\langle 0\rangle,\langle X_0\rangle,\langle X_0,X_1\rangle,\langle X_0,X_1,X_2-\mu\rangle,P_\la,\langle X_0Y_2\rangle\ |\ \mu\in\C,\la\in\C^{\times}\}
\end{align*}
where the ideals $P_\la$ are described in Proposition \ref{prop:Pspec_0}. Moreover only $\langle 0\rangle$ and $\langle X_0,X_1\rangle$ are not Poisson primitive. Further, $\Pspec A(2,a)$ is homeomorphic to $\Pspec A(2,b)$ if and only if $(a^2+a)/(b^2+b) > 0$.
\end{theorem}

\begin{proof}
The description of $\Pspec A(2,a)$ is immediate from Proposition~\ref{prop:Pspec_0}.
Pictorially, $\Pspec A(2,a)$ is given by the following diagrams, where  lines represent inclusion of ideals.

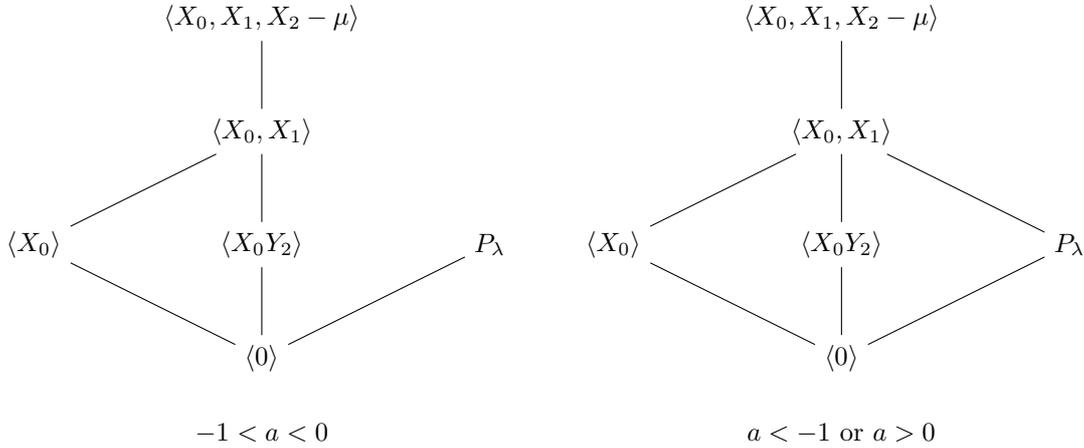
\begin{tabular}{ccc}
{		\begin{tikzpicture}
	\node (a) at (0,0){$\langle 0 \rangle$};
	\node (b1) at (0,1.5){$\langle X_0Y_2 \rangle$};
	\node (b2) at (-3,1.5){$\langle X_0\rangle$};
	\node (b3) at (3,1.5){$P_\la$};
	\node (c) at (0,3){$\langle X_0,X_1 \rangle$};
	\node (d) at (0,4.5){$\langle X_0,X_1,X_2-\mu\rangle$};
	\draw [-] (a) to (b1); \draw [-] (a) to (b2); \draw [-] (a) to (b3);
	\draw [-] (b1) to (c); \draw [-] (b2) to (c); 
	\draw [-] (c) to (d);
	\node (e) at (0,-1) {$-1<a<0$};
		\end{tikzpicture}}
&&
{		\begin{tikzpicture}
	\node (a) at (0,0){$\langle 0 \rangle$};
	\node (b1) at (0,1.5){$\langle X_0Y_2 \rangle$};
	\node (b2) at (-3,1.5){$\langle X_0\rangle$};
	\node (b3) at (3,1.5){$P_\la$};
	\node (c) at (0,3){$\langle X_0,X_1 \rangle$};
	\node (d) at (0,4.5){$\langle X_0,X_1,X_2-\mu\rangle$};
	\draw [-] (a) to (b1); \draw [-] (a) to (b2); \draw [-] (a) to (b3);
	\draw [-] (b1) to (c); \draw [-] (b2) to (c); \draw [-] (b3) to (c);
	\draw [-] (c) to (d);
		\node (e) at (0,-1) {$a<-1$ or $a>0$};
		\end{tikzpicture}.}
\end{tabular}

Note that despite the fact that the algebras $A(2,a)$ are not isomorphic for different values of $a$ (see Theorem \ref{thm:isom}), their spectra fall generically into only two non homeomorphic families. 

Since the topology on $\Pspec A(2,a)$ is governed by inclusions, it is clear that $\Pspec A(2,a)$ and $\Pspec A(2,b)$ are homeomorphic if and only if $a,b$ fall into the same case of \eqref{chris}, which is if and only if $(a^2+a)/(b^2+b) >0$.\end{proof}		

\begin{example}
\label{exinc}
Let $a=1/2$. Then $p=-1$, $q=-2$, $d_2=1$, $u_2=-5$ and $v_2=1$. Thus
\[P_\la=\langle X_0(X_2-\la X_0^5)-1/2 X_1^2\rangle\subseteq\langle X_0,X_1\rangle.\]
On the other hand when $a=-1/2$ we have $p=-1$, $q=2$, $d_2=1$, $u_2=3$ and $v_2=1$. Thus
\[P_\la=\langle X_0^2(X_0X_2-1/2 X_1^2)-\la\rangle,\]
which is Poisson maximal. Therefore the two spectra are not homeomorphic.
\end{example}

\begin{remark}\label{rem:viva}
From Theorem~\ref{thm:spectra2}, when $a \not\in \QQ$ then 
\[\Pspec A(2,a) = \{ \ang{0}, \ang{X_0}, \ang{X_0Y_2}, \ang{X_0, X_1}, \ang{X_0, X_1, X_2 - \mu}\ |\ \mu \in \kk\}.\]
Note that these are the prime ideals that are Poisson for all values of $a$.  

For all $a \neq -1,0$ there are five $d$-graded Poisson primes of $A(2,a)$: the ideals 
\[ \{ \ang{0}, \ang{X_0}, \ang{X_0Y_2}, \ang{X_0, X_1}, \ang{X_0, X_1, X_2 } \}.\]
\end{remark}

For the remainder of this section, we will study $\Pspec A(3,a)$ for various values of $a$.
	Again, we are most interested in $\Pspec_0 A(3,a)$ which is homeomorphic to $\Spec Z$ as above.
	From Lemma~\ref{lem:freeZ} we know that $Z$ is a free module over $Z'=\C[Y_2',Y_3']$ with basis
\[S=\{X_0^{s}Y_2^{s_2}Y_3^{s_3}\ |\ 0\leq s_i<v_i,\ i=2,3\ \text{and}\ ps+u_2d_2s_2+u_3d_3s_3=0\}.\]

\begin{proposition}\label{prop:al=1}
The cardinality of $S$ is equal to  $\al:=\frac{-p}{d_2d_3}$.
\end{proposition}
\begin{proof}
Since $p=-d_2v_2=-d_3v_3$ the equation $ps+u_2d_2s_2+u_3d_3s_3=0$ implies that $d_2$ divides $s_3$ (since $\gcd(d_2,u_3d_3)=1$) and $d_3$ divides $s_2$ (since $\gcd(d_3,u_2d_2)=1$). Set $s_2:=k_2d_3$ and $s_3:=k_3d_2$. Then the equation $ps+u_2d_2s_2+u_3d_3s_3=0$ becomes
\begin{equation}
\label{dio}
-\al s+u_2k_2+u_2k_2=0.
\end{equation}
To solve this Diophantine equation we first set $t=-\al s+u_2k_2$ and we solve $t+u_3k_3=0$. We get 
\[\begin{cases} t=-u_3k, \\ k_3=k, \end{cases}\]
for $k\in\Z$. We now solve the Diophantine equation $-\al s+u_2k_2=t=-u_3k$. Since $\gcd(-\al,u_2)=1$ there exist $m,n\in\Z$ such that $-\al m+u_2n=1$. The solution of the equation $-\al s+u_2k_2=-u_3k$ is then
\[\begin{cases} s=-mu_3k+u_2\ell, \\ k_2=-nu_3k+\al \ell, \end{cases}\]
for $k,\ell\in\Z$, and the solution of (\ref{dio}) is
\begin{equation}
\label{soldio}
\begin{cases} s=-mu_3k+u_2\ell, \\ k_2=-nu_3k+\al \ell , \\k_3=k, \end{cases}
\end{equation}
for $k,\ell\in\Z$.

For $i=2,3$ we have $0\leq k_i<\al=\frac{-p}{d_2d_3}$ since $0\leq s_i<v_i$. Fix $k_3\in\{0,\dots,\al-1\}$. Since $k_2=-nu_3k_3+\al \ell$ for some $\ell\in\Z$, there exists a unique $\ell\in\Z$ such that $k_2\in\{0,\dots,\al-1\}$. The integer $s$ is uniquely determined by $k_2$ and $k_3$, so we conclude that to each $k_3\in\{0,\dots,\al-1\}$ corresponds a unique monomial $X_0^{s}Y_2^{s_2}Y_3^{s_3}\in S$. 
\end{proof}

Since $\al=1\iff S=\{1\}$ we deduce the following corollary.

\begin{corollary}
We have $Z=Z'$ if and only if $\al=1$ if and only if $p\in\{-1,-2,-3,-6\}$. \qed
\end{corollary}

From the proof of Proposition \ref{prop:al=1} we observe that
\[S=\{X_0^{s}Y_2^{d_3k_2}Y_3^{d_2k_3}\ |\ s\in\Z,\ 0\leq k_i<\al,\ \text{and}\ -\al s+u_2k_2+u_3k_3=0\},\]
and that a recipe for finding these basis elements consists of computing $m,m'$ such that $-\al m+u_2m'=1$, plugging these values into (\ref{soldio}) and, for each $0\leq k_3<\al$, finding the unique $l\in\Z$ such that $k_2=-m'u_3k_3+\al l\in\{0,\dots,\al-1\}$.
We conclude this section with a couple of examples with $\al\neq1$ which illustrate possibilities for the ring structure of $Z$. But first we describe the $d$-graded Poisson prime that are common to $A(3,a)$ for any generic $a$. By Theorem~\ref{thm:spectra2} we have $\Pspec_{\gr} A^\circ\cong \PP_k(2,3) \sqcup \{ (Y_2, Y_3) \}$. In particular, for any $[\alpha: \beta]  \in \PP^1$, the element
\[ F_{[\alpha:\beta] }= \alpha (Y_2 X_0^{-1})^3 + \beta (Y_3 X_0^{-1})^2\]
is Poisson central.  
Multiplying by $X_0^6$, we obtain a pencil of Poisson normal sextic elements of $A(4,a)$
\begin{multline*} 
X_0^6F_{[\alpha:\beta]} = 
\alpha(X_0^3 X_2^3 - \frac{3}{2} X_0^2 X_1^2 X_2^2 + \frac{3}{4} X_0 X_1^4 X_2- \frac{1}{8} X_1^6) \\
+ \beta(X_0^4 X_3^2 + X_0^2X_1^2 X_2^2+ \frac{1}{9} X_1^6-2 X_0^3X_1 X_2X_3+\frac{2}{3} X_0^2X_1^3 X_3-\frac{2}{3} X_0 X_1^4 X_2).
\end{multline*}
Therefore the $d$-graded Poisson prime ideals that are common to $A(3,a)$ for any $a\neq -2,-1,0$ are
\begin{align*} 
\{\ang{0}, \ang{X_0}, \ang{X_0Y_2}, \ang{X_0^2Y_3},&\ \ang{ X_0^6F_{[\alpha:\beta]}}, \ang{X_0, X_1}, \ang{X_0Y_2,X_0^2Y_3},\\
		&\ang{X_0, X_1, X_2 }, \ang{X_0, X_1, X_2, X_3}\ |\ [\alpha:\beta]\in\PP^1 \}.
\end{align*}
Among them the Poisson primitive ideals are $\ang{X_0, X_1}, \ang{X_0Y_2,X_0^2Y_3}$ and $\ang{X_0, X_1, X_2 }, \ang{X_0, X_1, X_2, X_3}$.

\begin{example}
For $a=-5/4$ (the algebra in Example~\ref{pym}) we have $\al=5$, $Y_2'=X_0^{3}Y_2^{5}$ and $Y_3'=X_0^{7}Y_3^{5}$. Moreover
\[S=\{(X_0^2Y_2Y_3)^i\ |\ i=0,\dots,4\},\]
with $(X_0^2Y_2Y_3)^5=Y_2'Y_3'$. Thus $Z\cong \C[A,B,C]/\langle C^5=AB\rangle$.
\end{example}


\begin{example}
Choose $a=-24/5$. We have $\al=4$, $Y_2'=X_0^{-7}Y_2^{12}$ and $Y_3'=X_0^{-3}Y_2^{8}$. 
We have $S=\{1,C,D,E\}$, where $C:=X_0^{-4}Y_2^{3}Y_3^{6}$, $D:=X_0^{-5}Y_2^{6}Y_3^{4}$ and $E:=X_0^{-6}Y_2^{9}Y_3^{2}$. We further set $A:=Y_2'$ and $B:=Y_3'$. 
We have
\[Z\cong\frac{\C[A,B,C,D,E]}{\langle C^2=BD,D^2=AB=CE,CD=BE,E^2=AD,DE=AC\rangle}.\]
 \end{example}


\bibliographystyle{amsalpha}


\begin{thebibliography}{ATVdB90}

\bibitem[AS87]{AS}
M.~Artin and W.~F.~Schelter, \emph{Graded algebras of global
  dimension {$3$}}, Adv. Math. \textbf{66} (1987), no.~2, 171--216.
  \MR{917738}

\bibitem[ATVdB90]{ATV}
M.~Artin, J.~Tate, and M.~Van~den Bergh, \emph{Some algebras associated to
  automorphisms of elliptic curves}, The {G}rothendieck {F}estschrift, {V}ol.\
  {I}, Progr. Math., vol.~86, Birkh\"auser Boston, Boston, MA, 1990,
  pp.~33--85. \MR{1086882}

\bibitem[AVdB90]{AV}
M.~Artin and M.~Van~den Bergh, \emph{Twisted homogeneous coordinate rings}, J.
  Algebra \textbf{133} (1990), no.~2, 249--271. \MR{1067406}

\bibitem[AZ01]{AZ}
M.~Artin and J.~J. Zhang, \emph{Abstract {H}ilbert schemes}, Algebr. Represent.
  Theory \textbf{4} (2001), no.~4, 305--394. \MR{1863391}

\bibitem[BG02]{BG}
K.~A. Brown and K.~R. Goodearl, \emph{Lectures on algebraic quantum groups},
  Advanced Courses in Mathematics. CRM Barcelona, Birkh\"auser Verlag, Basel,
  2002. \MR{1898492}

\bibitem[BGR73]{Bo}
W.~Borho, P.~Gabriel, and R.~Rentschler, \emph{Primideale in
  {E}inh\"ullenden aufl\"osbarer {L}ie-{A}lgebren ({B}eschreibung durch
  {B}ahnenr\"aume)}, Lecture Notes in Mathematics, Vol. 357, Springer-Verlag,
  Berlin-New York, 1973. \MR{0376790}

\bibitem[Bj{\"o}89]{Bjork}
J.~Bj{\"o}rk, \emph{The {A}uslander condition on {N}oetherian rings},
  S\'eminaire d'{A}lg\`ebre {P}aul {D}ubreil et {M}arie-{P}aul {M}alliavin,
  39\`eme {A}nn\'ee ({P}aris, 1987/1988), Lecture Notes in Math., vol. 1404,
  Springer, Berlin, 1989, pp.~137--173. \MR{1035224}

\bibitem[BLLM]{BLLM}
J.~Bell, S.~Launois, O.~Leon~Sanchez, and R.~Moosa, \emph{Poisson algebras via
  model theory and differential-algebraic geometry}, to appear in Journal of
  the European Mathematical Society.

\bibitem[BZ17]{BZ}
J.~Bell and J.~J. Zhang, \emph{An isomorphism lemma for graded rings}, 
  Proc. Amer. Math. Soc. \textbf{145}, 2017, no.~3, 989--994. \MR{3589298}

\bibitem[CGG89]{CGG}
V.~Coll, M.~Gerstenhaber, and A.~Giaquinto, \emph{An explicit deformation
  formula with noncommuting derivations}, Ring theory 1989 ({R}amat {G}an and
  {J}erusalem, 1988/1989), Israel Math. Conf. Proc., vol.~1, Weizmann,
  Jerusalem, 1989, pp.~396--403. \MR{1029329}

\bibitem[Fre13]{Freudenburg}
G.~Freudenburg, \emph{Foundations of invariant theory for the down operator},
  J. Symbolic Comput. \textbf{57} (2013), 19--47. \MR{3066449}

\bibitem[Gin06]{Ginzburg}
V.~Ginzburg, \emph{{C}alabi-{Y}au algebras}, arxiv:math/0612139, 2006.

\bibitem[GK66]{GK}
I.~M.~Gelfand and A.~A.~Kirillov, \emph{Sur les corps li\'es aux alg\`ebres
  enveloppantes des alg\`ebres de {L}ie}, Inst. Hautes \'Etudes Sci. Publ.
  Math. (1966), no.~31, 5--19. \MR{0207918}

\bibitem[Goo06]{Goodearl}
K.~R.~Goodearl, \emph{A {D}ixmier-{M}oeglin equivalence for {P}oisson algebras
  with torus actions}, Algebra and its applications, Contemp. Math., vol. 419,
  Amer. Math. Soc., Providence, RI, 2006, pp.~131--154. \MR{2279114}

\bibitem[GW89]{GW}
K.~R. Goodearl and R.~B. Warfield, Jr., \emph{An introduction to noncommutative
  {N}oetherian rings}, London Mathematical Society Student Texts, vol.~16,
  Cambridge University Press, Cambridge, 1989. \MR{1020298}

\bibitem[Har77]{Hartshorne}
R.~Hartshorne, \emph{Algebraic geometry}, Springer-Verlag, New
  York-Heidelberg, 1977, Graduate Texts in Mathematics, No. 52. \MR{0463157}

\bibitem[IS80]{IS}
R.~S. Irving and L.~W. Small, \emph{On the characterization of primitive
  ideals in enveloping algebras}, Math. Z. \textbf{173} (1980), no.~3,
  217--221. \MR{592369}

\bibitem[Jor14]{Jordan}
D.~A. Jordan, \emph{Ore extensions and {P}oisson algebras}, Glasgow Math. J.
  \textbf{56} (2014), no.~2, 355--368. \MR{3187902}

\bibitem[Jos77]{Jo}
A.~Joseph, \emph{A generalization of the {G}elfand-{K}irillov conjecture},
  Amer. J. Math. \textbf{99} (1977), no.~6, 1151--1165. \MR{0460397}

\bibitem[KRS05]{KRS}
D.~S. Keeler, D.~Rogalski, and J.~T. Stafford, \emph{Na\"\i ve noncommutative
  blowing up}, Duke Math. J. \textbf{126} (2005), no.~3, 491--546. \MR{2120116}

\bibitem[LBS93]{LBS}
L.~Le~Bruyn and S.~P.~Smith, \emph{Homogenized {${\germ s}{\germ l}(2)$}},
  Proc. Amer. Math. Soc. \textbf{118} (1993), no.~3, 725--730. \MR{1136235}

\bibitem[Lev92]{Lev}
T.~Levasseur, \emph{Some properties of noncommutative regular graded
  rings}, Glasgow Math. J. \textbf{34} (1992), no.~3, 277--300. \MR{1181768}

\bibitem[McC74]{MC}
J.~C.~McConnell, \emph{Representations of solvable {L}ie algebras and the
  {G}elfand-{K}irillov conjecture}, Proc. London Math. Soc. (3) \textbf{29}
  (1974), 453--484. \MR{0357529}

\bibitem[MR01]{MR}
J.~C.~McConnell and J.~C.~Robson, \emph{Noncommutative {N}oetherian rings},
  revised ed., Graduate Studies in Mathematics, vol.~30, American Mathematical
  Society, Providence, RI, 2001, With the cooperation of L. W. Small.
  \MR{1811901}

\bibitem[Oh99]{Oh}
S.~Q.~Oh, \emph{Symplectic ideals of {P}oisson algebras and the {P}oisson
  structure associated to quantum matrices}, Comm. Algebra \textbf{27} (1999),
  no.~5, 2163--2180. \MR{1683857}

\bibitem[Pym15]{Pym}
B.~Pym, \emph{Quantum deformations of projective three-space}, Adv. Math.
  \textbf{281} (2015), 1216--1241. \MR{3366864}

\bibitem[Ric02]{Richard}
L.~Richard, \emph{Equivalence rationnelle et homologie de {H}ochschild pour
  certaines alg\`{e}bres polynomiales classiques et quantiques}, Ph.D. thesis,
  Universit\'e Blaise Pascal, 2002.

\bibitem[Rog14]{Rogalski}
D.~Rogalski, \emph{An introduction to noncommutative projective geometry},
  arxiv: 1403.3065, 2014.

\bibitem[RRZ14]{RRZ}
M.~Reyes, D.~Rogalski, and J.~J.~Zhang, \emph{Skew {C}alabi-{Y}au
  algebras and homological identities}, Adv. Math. \textbf{264} (2014),
  308--354. \MR{3250287}

\bibitem[Shi05]{Shirikov}
E.~N.~Shirikov, \emph{Two-generated graded algebras}, Algebra Discrete
  Math. (2005), no.~3, 60--84. \MR{2237896}

\bibitem[Sig84]{Sigurdsson}
G.~Sigurdsson, \emph{Differential operator rings whose prime factors have
  bounded {G}oldie dimension}, Arch. Math. (Basel) \textbf{42} (1984), no.~4,
  348--353. \MR{753356}

\bibitem[ST01]{SheltonTingey}
B.~Shelton and C.~Tingey, \emph{On {K}oszul algebras and a new
  construction of {A}rtin-{S}chelter regular algebras}, J. Algebra \textbf{241}
  (2001), no.~2, 789--798. \MR{1843325}

\bibitem[vdE93]{vdEssen}
A.~van~den~Essen, \emph{An algorithm to compute the invariant ring of a
  {${\bf G}\sb a$}-action on an affine variety}, J. Symbolic Comput.
  \textbf{16} (1993), no.~6, 551--555. \MR{1279532}

\bibitem[Zha96]{Zhang}
J.~J. Zhang, \emph{Twisted graded algebras and equivalences of graded
  categories}, Proc. London Math. Soc. (3) \textbf{72} (1996), no.~2, 281--311.
  \MR{1367080}

\end{thebibliography}

\providecommand{\bysame}{\leavevmode\hbox to3em{\hrulefill}\thinspace}
\providecommand{\MR}{\relax\ifhmode\unskip\space\fi MR }
\providecommand{\MRhref}[2]{%
  \href{http://www.ams.org/mathscinet-getitem?mr=#1}{#2}
}
\providecommand{\href}[2]{#2}

\end{document}